\documentclass[a4paper,leqno,11pt]{amsart}
\usepackage[top=35mm, bottom=35mm, left=30mm, right=30mm]{geometry}
\usepackage{mathrsfs,amsfonts,amssymb,mathtools}
\usepackage{dsfont}
\usepackage{enumitem}
\usepackage{tikz}
\usepackage{verbatim}
\usetikzlibrary{arrows}
\usetikzlibrary{patterns,decorations.pathreplacing}
\usepackage[T1]{fontenc}

\usepackage{graphicx,float}
\usepackage[english]{babel}

\usepackage{color}

\usepackage[hidelinks]{hyperref}
\usepackage{nameref}
\hypersetup{colorlinks = true, urlcolor = blue, linkcolor = blue, citecolor = red}

\usepackage{array}
\numberwithin{equation}{section}
\newtheorem{theorem}[equation]{Theorem}
\newtheorem{lemma}[equation]{Lemma}
\newtheorem{proposition}[equation]{Proposition}

\theoremstyle{definition}
\newtheorem{definition}[equation]{Definition}
\newtheorem{remark}[equation]{Remark}
\newtheorem{example}[equation]{Example}
\newtheorem{corollary}[equation]{Corollary}

\newcommand{\RR}{\mathbb{R}}
\newcommand{\ZZ}{\mathbb{Z}}

\newcommand{\NN}{\mathbb{N}}
\newcommand{\TT}{\mathbb{T}}

\newcommand{\BBB}{\mathcal{B}}
\newcommand{\MMM}{\mathcal{M}}
\newcommand{\RRR}{\mathcal{R}}
\newcommand{\SSS}{\mathcal{S}}
\newcommand{\XXX}{\mathcal{X}}

\newcommand{\tom}{\mathbb{T}^\omega}

\newcommand{\ind}[1]{{\mathds{1}_{{#1}}}}

\makeatletter
\def\namedlabel#1#2{\begingroup#2%
	\def\@currentlabel{#2}%
	\phantomsection\label{#1}\endgroup}
\makeatother

\newcommand{\config}[6]{
	\draw[#6,thick,fill=#6,fill opacity=0.1] 
	(#1,#3) -- 
	(#1*#5*0.5+#2-#2*#5*0.5,#3) -- 
	(#1*#5*0.5+#2-#2*#5*0.5,#3*0.5+#4*0.5) -- (#1*0.5+#2*0.5,#3*0.5+#4*0.5) -- 
	(#1*0.5+#2*0.5,#3*#5*0.5+#4-#4*#5*0.5) -- 
	(#1,#3*#5*0.5+#4-#4*#5*0.5) -- 
	(#1,#3);
	\draw[#6,draw=none,fill=#6,fill opacity=0.2] 
	(#1,#3*#5*0.5+#3*0.5+#4*0.5-#4*#5*0.5) rectangle 
	(#1*0.5+#2*0.5,#3*0.5+#4*0.5);
	\draw[#6,draw=none,fill=#6,fill opacity=0.2] 
	(#1*#5*0.5+#1*0.5+#2*0.5-#2*#5*0.5,#3) rectangle 
	(#1*0.5+#2*0.5,#3*0.5+#4*0.5);
	\draw[#6,draw=none,fill=#6,fill opacity=0.2] 
	(#1*#5*0.5+#1*0.5+#2*0.5-#2*#5*0.5,#3*#5*0.5+#3*0.5+#4*0.5-#4*#5*0.5) rectangle 
	(#1*0.5+#2*0.5,#3*0.5+#4*0.5);
	\draw[#6,draw=none] 
	(#1,#3) rectangle (#1*0.5+#2*0.5,#3*0.5+#4*0.5); 
	\draw[#6,draw=none] 
	(#1,#3*#5*0.5+#3*0.5+#4*0.5-#4*#5*0.5) rectangle (#1*0.5+#2*0.5,#3*#5*0.5+#4-#4*#5*0.5); 
	\draw[#6,draw=none] 
	(#1*#5*0.5+#1*0.5+#2*0.5-#2*#5*0.5,#3) rectangle (#1*#5*0.5+#2-#2*#5*0.5,#3*0.5+#4*0.5);
	\draw[#6,dashed] (#1*0.5+#2*0.5,#3) -- (#1*0.5+#2*0.5,#3*0.5+#4*0.5);
	\draw[#6,dashed] (#1,#3*0.5+#4*0.5) -- (#1*0.5+#2*0.5,#3*0.5+#4*0.5);
	\draw[#6,dashed] (0.5*#1+0.5*#2+0.5*#1*#5-0.5*#2*#5,#3) -- (0.5*#1+0.5*#2+0.5*#1*#5-0.5*#2*#5,#3*0.5+#4*0.5);
	\draw[#6,dashed] (#1,0.5*#3+0.5*#4+0.5*#3*#5-0.5*#4*#5) -- (#1*0.5+#2*0.5,0.5*#3+0.5*#4+0.5*#3*#5-0.5*#4*#5);
}

\newcommand{\quarters}[4]{
	\draw[thin,opacity=0.25] (#1,#3) rectangle (#2,#4);
	\draw[dashed,opacity=0.25] (#1,0.5*#3+0.5*#4) -- (#2,0.5*#3+0.5*#4);
	\draw[dashed,opacity=0.25] (0.5*#1+0.5*#2,#3) -- (0.5*#1+0.5*#2,#4);
}

\begin{document}
	
	\title[On differentiation of integrals in Lebesgue spaces]{On differentiation of integrals in Lebesgue spaces}
	
	\author{Marco Fraccaroli}
	\address{Marco Fraccaroli (\textnormal{mfraccaroli@bcamath.org})
		\newline BCAM -- Basque Center for Applied Mathematics, 48009 Bilbao, Spain}		
	\author{Dariusz Kosz}
	\address{Dariusz Kosz (\textnormal{dariusz.kosz@pwr.edu.pl})
		\newline Wroc\l aw University of Science and Technology, 50-370 Wroc\l aw, Poland}
	\author{Luz Roncal}
	\address{Luz Roncal (\textnormal{lroncal@bcamath.org})
		\newline BCAM -- Basque Center for Applied Mathematics, 48009 Bilbao, Spain \newline
		Universidad del Pa{\'i}s Vasco / Euskal Herriko Unibertsitatea, 48080 Leioa, Spain \newline
		Ikerbasque, Basque Foundation for Science, 48011 Bilbao, Spain}		
	\thanks{Marco Fraccaroli and Luz Roncal were supported by the Basque Government through the BERC 2022-2025 program and by the Ministry of Science and Innovation: BCAM Severo Ochoa accreditation CEX2021-001142-S/MICIN/AEI/\allowbreak10.13039/501100011033 and PID2023-146646NB-I00 funded by MICIU/\allowbreak AEI/10.13039/501100011033 and by ESF+. Dariusz Kosz was supported by the National Science Centre of Poland grant SONATA BIS 2022/46/E/ST1/00036. Luz Roncal was also supported by IKERBASQUE and by the Ministry of Science and Innovation through CNS2023-143893. \\
		\indent The authors gratefully acknowledge the support provided during the Simons Semester \emph{Contemporary Harmonic Analysis and its Synergies}, held in Wroc{\l}aw, B\k{e}dlewo, and Warsaw from August 1 to September 30, 2024; this event was part of the Simons Semesters at the Banach Center, funded by the Simons Foundation (Award Number: 663281).
	}
	
	\begin{abstract} We study the problem of differentiation of integrals for certain bases in the infinite-dimensional torus $\mathbb{T}^\omega$. In particular, for every $p_0 \in [1,\infty)$, we construct a basis $\mathcal{B}$ which differentiates $L^p(\mathbb{T}^\omega)$ if and only if $p \geq p_0$, thus reproving classical theorems of Hayes in $\mathbb{R}$. The main novelty is that our $\mathcal{B}$ is a Busemann--Feller basis consisting of rectangles (of arbitrarily large dimensions) with sides parallel to the coordinate axes. Our construction gives us the opportunity to classify all possible ranges of differentiation for general complete spaces.
		Namely, let $\BBB$ be a basis in a metric measure space $\XXX$. If $\XXX$ is complete, then the set $\{ p \in [1,\infty] : \BBB \text{ differentiates } L^p(\XXX) \}$ takes one of the six forms 
		\[
		\emptyset, \, \{\infty\}, \, [p_0,\infty], \, (p_0,\infty], \, [p_0,\infty), \, (p_0,\infty)
		\quad \text{for some} \quad p_0 \in [1,\infty).
		\]
		Conversely, for every $p_0 \in [1,\infty)$ and each of the six cases above, we construct a complete space $\XXX$ and a basis $\BBB$ illustrating the corresponding range of differentiation.     
		
		
	\end{abstract}
	
	\date{\today}
	\maketitle
	
	
	\section{Introduction}\label{S1}
	
	\subsection{Statements of the results}\label{S1.1} The problem of differentiation of integrals concerns the question of whether, or under what conditions, the value of a given function at a certain point is the limit of mean value integrals of the function over neighborhoods approximating the point. One of the classical results in this area is the celebrated Lebesgue differentiation theorem. It asserts that, for every locally integrable function $f \colon \RR^d \to \RR$, the value $f(x)$ can be recovered for almost every $x \in \RR^d$ as the limit of the averages of $f$ over the balls $B(x,r)$ as $r \to 0$. While the same holds for centered cubes in place of balls, the case of centered rectangles with sides parallel to the coordinate axes and diameters tending to $0$ requires a slightly stronger assumption on $f$. For example, the local integrability of $|f|^p$ for some $p \in (1,\infty)$ is sufficient. These three results also have uncentered counterparts, where $x$ belongs to each given set but is not necessarily its center.
	
	In this article, we explore the problem of differentiation of integrals in the context of the infinite-dimensional torus $\tom$ (for a detailed description of this space, see Section~\ref{S4}), focusing on certain families of uncentered rectangles with sides parallel to the coordinate axes. In particular, we construct families for which the condition that $|f|^p$ is integrable is sufficient for the corresponding variant of the Lebesgue differentiation theorem when $p \geq p_0$ (resp. $p > p_0$) but not when $p < p_0$ (resp. $p \leq p_0$).
	
	Our first main result reads as follows (for the notation used, see Subsection~\ref{S1.3}).
	
	\begin{theorem} \label{T0}
		Fix $p_0 \in [1,\infty)$. Then there exist Busemann--Feller bases $\BBB_\geq$ and $\BBB_>$ in the infinite-dimensional torus $\tom$, whose elements are rectangles with sides parallel to the coordinate axes, 
		such that for every $p \in [1,\infty]$ we have the following:
		\begin{itemize}
			\item $\BBB_\geq$ differentiates $L^p(\tom)$ if and only if $p \geq p_0$,
			\item $\BBB_>$ differentiates $L^p(\tom)$ if and only if $p > p_0$.
		\end{itemize}    
	\end{theorem}
	
	A couple of comments regarding Theorem~\ref{T0} are in order.
	\begin{enumerate}
		\item Hayes \cite{Ha52, Ha52a, Ha58} provided examples of differentiation bases in the one-dimensional torus $\TT$ with similar properties. For example, he constructed a basis that differentiates $L^p(\TT)$ if and only if $p \geq p_0$. The main novelty of Theorem~\ref{T0} is that our examples are Busemann--Feller bases consisting of rectangles with sides parallel to the coordinate axes---something impossible in any finite-dimensional setting. For a more detailed discussion, see Subsection~\ref{S1.2}.   
		\item In the proof of Theorem~\ref{T0}, we explicitly construct families of rectangles so that the associated bases satisfy the given properties. Our examples are largely inspired by the objects studied in the series of articles \cite{FR20, Ko21, KMPRR23, KRR23}. Moreover, they can be used to find bases in the spirit of Hayes, see Corollary~\ref{Cor1}.  
	\end{enumerate}
	
	\begin{corollary} \label{Cor1}
		Fix $p_0 \in [1,\infty)$. Then there exist Busemann--Feller bases $\BBB_\geq^*$ and $\BBB_>^*$ in the unit interval $[0,1]$, whose elements are finite unions of open intervals, such that  
		$\BBB_\geq^*$ (resp.~$\BBB_>^*$) differentiates $L^p([0,1])$ if and only if $p \geq p_0$ (resp.~$p > p_0$).  
	\end{corollary}
	
	With Theorem~\ref{T0} in hand, we turn our attention to a related problem concerning the Lebesgue spaces $L^p(\XXX)$ for arbitrary metric measure spaces $\XXX = (X, \rho, \mu)$, and arbitrary differentiation bases $\BBB$. Namely, given a pair $(\XXX,\BBB)$ and $p_1,p_2 \in [1,\infty]$ such that $p_1 < p_2$, we ask whether there is a relationship between the variants of the Lebesgue differentiation theorem formulated for $L^{p_1}(\XXX)$ and $L^{p_2}(\XXX)$. Of course, the answer is clear when the space has finite measure (or when the problem reduces to such a case), since then $L^{p_2}(\XXX) \subseteq L^{p_1}(\XXX)$. Here, however, we only assume that $\XXX$ is complete, and interestingly, in this more general setting, the answer depends on whether $p_2$ is finite or not.
	
	Our second main result reads as follows.
	
	\begin{theorem}\label{T1}
		Let $\BBB$ be a differentiation basis in $\XXX = (X, \rho, \mu)$. Define
		\[
		{\rm diff}(\BBB) \coloneqq
		\big\{ p \in [1,\infty] : \BBB \text{ differentiates } L^p(\XXX) \big\}.
		\]
		If $\XXX$ is complete, then, for some $p_0 \in [1,\infty)$, one of the following six cases occurs:
		\begin{enumerate}
			\item[\namedlabel{C1}{\rm(C1)}] 
			${\rm diff}(\BBB) = \emptyset$,
			
			\item[\namedlabel{C2}{\rm(C2)}] 
			${\rm diff}(\BBB) = \{\infty\}$,
			
			\item[\namedlabel{C3}{\rm(C3)}] 
			${\rm diff}(\BBB) = [p_0, \infty]$,
			
			\item[\namedlabel{C4}{\rm(C4)}]
			${\rm diff}(\BBB) = (p_0, \infty]$,
			
			\item[\namedlabel{C5}{\rm(C5)}] ${\rm diff}(\BBB) = [p_0, \infty)$,
			
			\item[\namedlabel{C6}{\rm(C6)}] ${\rm diff}(\BBB) = (p_0, \infty)$.
		\end{enumerate}
		Moreover, if $X$ is a countable union of (possibly overlapping) open sets with finite measure, then (regardless of whether $\XXX$ is complete) one of the four cases \ref{C1}--\ref{C4} occurs. 
		
		Conversely, fix $p_0 \in [1,\infty)$ and $P \subseteq [1,\infty]$. If $P$ takes one of the forms \ref{C1}--\ref{C6}, then one can find a pair $(\XXX, \BBB)$ such that ${\rm diff}(\BBB) = P$. Moreover, if $P$ takes one of the forms \ref{C1}--\ref{C4}, then one can find a pair $(\XXX, \BBB)$ such that ${\rm diff}(\BBB) = P$ and $\mu(X) < \infty$.
	\end{theorem}
	
	Several comments regarding Theorem~\ref{T1} are in order.
	\begin{enumerate} 
		\item The first part of Theorem~\ref{T1} follows from a monotonicity property stating that the larger $p$ is, the better the differentiation properties of $\BBB$ are (see Proposition~\ref{P1}). We verify this property for $p \in [1,\infty)$ assuming that $\XXX$ is complete, or for all $p \in [1,\infty]$ assuming that $X$ is a countable sum of open sets with finite measure. 
		\item Regarding the case $p \in [1,\infty)$ above, we assume that $\XXX$ is complete in order to better handle pathological measurability issues (see Example~\ref{E2} and Example~\ref{E3}). This is a very mild restriction, since there is a canonical way to extend any given measure space to a complete one. However, we do not rule out the possibility that this assumption can be completely removed. 
		On the other hand, we know that, in general, the range $[1,\infty)$ cannot be extended to $[1,\infty]$ (see Example~\ref{E1}). 
		\item The second part of Theorem~\ref{T1} follows by combining Example~\ref{E1} with the family of examples constructed in the proof of Theorem~\ref{T0}.
		\item We also study the relationship between the range ${\rm diff}(\BBB)$ for a given uncentered basis $\BBB$ and the behavior of the associated maximal operator (see Remark~\ref{rem3}).  
	\end{enumerate}
	
	\subsection{Historical background and related results}\label{S1.2}
	
	In \cite{Ha52, Ha52a, Ha58}, Hayes provided examples of differentiation bases $\BBB^*$ in the one-dimensional torus $\mathbb{T}$ exhibiting the cases \ref{C2}--\ref{C3}, see also \cite[Chapters~VI.1 and VII.5]{dG75}. A minor adjustment of Hayes' construction is likely to yield the case \ref{C4} as well, introducing logarithmic factors as in our example for \ref{C4}. 
	
	It is worth noting that the elements of Hayes' differentiation bases $\BBB^*$ are not intervals, but disjoint unions of intervals. In fact, in the $d$-dimensional torus $\mathbb{T}^d$ for any fixed $d \in \NN$, no Busemann--Feller basis $\BBB$ made of arbitrary $d$-dimensional rectangles with sides parallel to the axes can exhibit the cases \ref{C1}--\ref{C4} for any $p_0 \in (1,\infty)$. This is due to the $L^p(\RR^d)$ bounds for the strong maximal operator for every $p \in (1,\infty]$, see \cite{JMZ35, Fa72, CF75}. As a consequence, for each such $\BBB$ we would have ${\rm diff}(\BBB) \in \{[1, \infty], (1, \infty] \}$.
	
	The main novelty of the article is that we circumvent this finite-dimensional obstruction by considering differentiation bases of rectangles in $\mathbb{T}^{\omega}$, recovering all the cases \ref{C1}--\ref{C4}. 
	
	Let us mention that in \cite{HS09}, Hagelstein and Stokolos proved that, for any fixed $d \in \NN$ and for every homothecy invariant density basis $\BBB$, whose elements are convex sets in $\RR^d$, one of the cases \ref{C3}--\ref{C4} occurs for some $p_0 \in [1,\infty)$. They also showed in \cite{HS24} that, if $d=2$, then a previous result of Bateman \cite{Ba09} implies that necessarily $p_0 = 1$; not long after, however, together with Radillo-Murgu{\'i}a \cite{HRS24} (see also \cite{HRS24a}), they identified a gap in Bateman’s proof, providing a counterexample to \cite[Claim~7(B)]{Ba09}. This makes the result from \cite{HS24} conditional upon the correctness of the main theorem in \cite{Ba09}. In this regard, we point out that the bases we exhibit in the proof of Theorem~\ref{T0} are far from being homothecy or even translation invariant, with the sole exception of the basis used to cover the case \ref{C1}. Therefore, both Theorem~\ref{T0} and Theorem~\ref{T1} help clarify the possible ranges of differentiation in the context of arbitrary bases, but they do not contribute to the corresponding problem formulated for homothecy invariant bases.
	
	We conclude this section recalling the characterization of the weak $L^p(\XXX)$ boundedness of a maximal function associated with a collection of subsets $\BBB$ in terms of the covering properties of $\BBB$ with $p_*$ such that $\frac{1}{p} + \frac{1}{p_*} = 1$ (for the notation used, see Section~\ref{S3}).
	\begin{theorem}[\cite{CF75}]\label{char_max}
		Fix $p \in (1,\infty)$. Let $\MMM_\BBB$ be the maximal operator associated with a collection $\BBB$ in a $\sigma$-finite complete metric measure space $\XXX$. The following are equivalent:
		\begin{enumerate}
			\item[{\rm(1)}]
			$\MMM_\BBB$ satisfies the weak-type $(p,p)$ inequality,
			
			\item[{\rm(2)}] 
			$\BBB$ has the \emph{covering property $V_{p_*}$}, i.e., there is a constant $C \in (0,\infty)$ such that, for every collection $\mathcal{E} \subseteq \BBB$, there is a subcollection $\mathcal{F} \subseteq \mathcal{E}$ such that 
			\begin{equation*}
				\mu \Big( \bigcup_{E \in \mathcal{E}} E \Big) \leq C \mu \Big( \bigcup_{F \in \mathcal{F}} F \Big), \qquad \Big\| \sum_{F \in \mathcal{F}} \ind{F} \Big\|_{L^{p_*}(\XXX)}  \leq C \mu \Big( \bigcup_{E \in \mathcal{E}} E \Big)^{\frac{1}{p_*}}.
			\end{equation*}
		\end{enumerate}
	\end{theorem}
	An analogous characterization of the differentiability of $L^p(\XXX)$ is also available.
	\begin{theorem}[$p=\infty$ in \cite{dP36}, $p \in (1,\infty)$ in \cite{HP55, Ha76}]\label{char_diff}
		Fix $p \in (1,\infty]$. Let $\BBB$ be a differentiation basis in a $\sigma$-finite complete metric measure space $\XXX$. The following are equivalent:
		\begin{enumerate}
			\item[{\rm(1)}]
			$\BBB$ differentiates $L^p(\XXX)$,
			
			\item[{\rm(2)}] 
			$\BBB$ has the \emph{covering strength $\phi_{p_*}(x) = x^{p_*}$}, i.e., for every measurable set $E \subseteq X$ with finite measure, 
			every differentiation basis $\mathcal{E}$ of $E$, and every $\varepsilon  \in (0,\infty)$, there is a finite collection $\mathcal{F} \subseteq \mathcal{E}$ such that
			\begin{equation*}
				\mu \Big( E \setminus \bigcup_{F \in \mathcal{F}} F \Big) = 0, \qquad \mu \Big( \bigcup_{F \in \mathcal{F}} F \setminus E \Big) < \varepsilon, \qquad \Big\| \sum_{F \in \mathcal{F}} \ind{F} - \ind{\bigcup_{F \in \mathcal{F}} F} \Big\|_{L^{p_*}(\XXX)}  < \varepsilon.
			\end{equation*}
		\end{enumerate}
	\end{theorem}
	Under the additional assumption that $\BBB$ is translation invariant, Theorem~\ref{char_diff} for $p \in (1,\infty)$ was proved in \cite{Co76}. Furthermore, analogous characterizations for Orlicz spaces were implicit in the bounds for $\MMM_\BBB$ given in \cite{CF75}, and were explicitly exhibited in the study of the differentiation properties of $\BBB$ in \cite{Ha76a}. Nevertheless, the case of $L^1(\XXX)$ remains open (specifically, the implication $(1) \implies (2)$ is missing in the respective variants of both Theorem~\ref{char_max} and Theorem~\ref{char_diff}).
	
	Despite providing elegant geometric conditions, the characterizations may still be complicated to exploit, as the covering properties are difficult to assess. In the proof of Theorem~\ref{T1}, we will rely on the construction of explicit functions to show that $p \notin {\rm diff}(\BBB)$.
	
	We refer the interested reader to two surveys of Hagelstein \cite{Ha24, Ha24a} for more details about results, open problems, and references in the theory of differentiation of integrals.
	
	\subsection{Notation and basic definitions}\label{S1.3}
	Let $\XXX$ be a \textit{metric measure space}, that is, a triple $(X,\rho,\mu)$, where $X$ is a nonempty set, $\rho$ is a metric, and $\mu$ is a Borel measure. We say that $\XXX$ is \textit{complete} if every subset of a set of measure zero is itself measurable (and thus has measure zero). For every $x \in X$ and $r \in (0,\infty)$, we define the (open) ball centered at $x$ of radius $r$ as 
	\[
	B_\rho(x,r) \coloneqq \{ y \in X : \rho(x,y) < r \}.
	\]
	Later on, we omit the index $\rho$, since there is no risk of confusion. 
	
	For $p \in [1, \infty)$, the space $L^p(\XXX)$ consists of all measurable functions $f \colon X \to \RR$ such that the corresponding quantity
	\[
	\| f \|_{L^p(\XXX)} \coloneqq
	\Big( \int_X |f|^p \, {\rm d}\mu \Big)^{1/p}
	\]
	is finite. Similarly, $L^\infty(\XXX)$ consists of all measurable functions $f \colon X \to \RR$ such that
	\[
	\| f \|_{L^\infty(\XXX)} \coloneqq
	\inf \{ C \in [0,\infty] : |f(x)| \leq C \text{ for }\mu\text{-almost every } x \in X\}
	\]
	is finite. We identify two measurable functions if they are equal $\mu$-almost everywhere. We simply write $L^p(X)$ instead of $L^p(\XXX)$ when $\rho$ and $\mu$ are chosen canonically. 
	
	Suppose that, for every $x \in X$, we have a family $\BBB(x)$ of measurable sets $S \subseteq X$ such that $x \in S$ and $\mu(S) \in (0, \infty)$. A sequence \textit{$(S_n)_{n \in \NN}$ contracts to $x$} if $S_n \in \BBB(x)$ for every $n \in \NN$ and, moreover, there exists a sequence $(r_n)_{n \in \NN}$ of radii such that $\lim_{n \to \infty} r_n = 0$ and $S_n \subseteq B(x,r_n)$ for every $n \in \NN$; in each such case we write $(S_n)_{n \in \NN} \Rightarrow x$. 
	
	The whole family $\BBB \coloneqq \bigcup_{x \in X} \BBB(x)$ equipped with the relation $\Rightarrow$ is called a \textit{differentiation basis in $\XXX$} if for $\mu$-almost every $x \in X$ the set of sequences contracting to $x$ is nonempty. We say that $\BBB$ is \textit{uncentered} if 
	\[
	\BBB(x) = \{S \in \BBB : x \in S\}
	\]
	for every $x \in X$. If $\BBB$ is uncentered and, moreover, every $S \in \BBB$ is open, then $\BBB$ is called a \textit{Busemann--Feller basis} \cite[Chapter~II.2]{dG75}.
	
	Fix a pair $(\XXX, \BBB)$ and $p \in [1,\infty]$. We say that \textit{$\BBB$ differentiates $L^p(\XXX)$} if, for every fixed $f \in L^p(\XXX)$, the following condition 
	\begin{align} \label{D} \tag{D}
		\lim_{n \to \infty} {\rm Avg}_f(S_n) = f(x)
		\quad \text{for every} \quad
		(S_n)_{n \in \NN} \Rightarrow x
	\end{align}
	holds for $\mu$-almost every $x \in X$. 
	Here, for any $E \subseteq X$ satisfying $\mu(E) \in (0,\infty)$, we write ${\rm Avg}_f(E)$ to denote the average of $f$ over $E$, that is,
	\[
	{\rm Avg}_f(E) \coloneqq \frac{1}{\mu(E)} \int_E f \, {\rm d}\mu.
	\]
	The set of all $p \in [1,\infty]$ such that $\BBB$ differentiates $L^p(\XXX)$ is denoted by 
	${\rm diff}(\BBB)$. Furthermore, we say that $\BBB$ is a \emph{density basis} if it differentiates $\{  \ind{E} \colon \mu(E) \in [0,\infty] \}$, which is a strict subset of $L^\infty(\XXX)$.  
	
	For every $p \in [1,\infty]$ and $f \in L^p(\XXX)$, we define the exceptional set
	\begin{align*} 
		E(f) \coloneqq \{ x \in X : \overline{f}(x) \neq f(x) \text{ or } \underline{f}(x) \neq f(x) \}.
	\end{align*}
	Here $\overline{f} \colon X \to [-\infty, \infty]$ and $\underline{f} \colon X \to [-\infty, \infty]$ are defined respectively by
	\begin{align*}
		\overline{f}(x) & \coloneqq \max \big \{ c \in [-\infty,\infty] : \lim_{n \to \infty} {\rm Avg}_f(S_n) = c \text{ for some } (S_n)_{n \in \NN} \Rightarrow x \big\}, \\
		\underline{f}(x) & \coloneqq \, \min \big \{ c \in [-\infty,\infty] : \lim_{n \to \infty} {\rm Avg}_f(S_n) = c \text{ for some } (S_n)_{n \in \NN} \Rightarrow x \big\}. 
	\end{align*}
	Observe that $\overline{f}$ and $\underline{f}$ are measurable if $\BBB$ is uncentered and consists of either countably many sets or arbitrarily many open sets. Indeed, for every $\lambda \in \RR$ we have
	\[
	\{ x \in X : \overline{f}(x) \geq \lambda \} = \bigcap_{n \in \NN} \bigcup_{S \in \BBB_{\lambda, n}} S,  
	\]
	where $\BBB_{\lambda, n}$ consists of all $S \in \BBB$ such that ${\rm Avg}_f(S) \geq \lambda - 2^{-n}$ and $S \subseteq B$ for some ball $B \subseteq X$ of radius $2^{-n}$. 
	By symmetry, $\{ x \in X : \underline{f}(x) \leq \lambda \}$ can be described in terms of $S \in \BBB$ such that ${\rm Avg}_{-f}(S) \geq -\lambda - 2^{-n}$ and $S \subseteq B$ for some ball $B \subseteq X$ of radius $2^{-n}$. 
	
	\subsection{Structure of the article} \label{S1.4}
	The rest of the paper is organized as follows.
	
	In Section~\ref{S2}, we formulate a monotonicity property which justifies the first part of Theorem~\ref{T1} (see Proposition~\ref{P1}). We also comment on the special case $p=\infty$, discuss the measurability issues mentioned before, and present a few instructive examples.
	
	In Section~\ref{S3}, we study maximal operators and their role in the problem of differentiation of integrals. In particular, we observe that nestedness interacts well with weak-type maximal inequalities in a setting that extends beyond the realm of martingale-type bases. The resulting statement may be of independent interest (see Proposition~\ref{P2}).
	
	In Section~\ref{S4}, we deal with the infinite-dimensional torus $\tom$. We introduce a class of rectangular bases in $\tom$ that will be used to prove Theorem~\ref{T0} (see Proposition~\ref{P3}).
	
	In Section~\ref{S5}, we prove Theorem~\ref{T0}, Corollary~\ref{Cor1}, and Theorem~\ref{T1}.
	
	\subsection*{Acknowledgments} We thank Paul Hagelstein and Ioannis Parissis for reading the first version of the manuscript, and for providing valuable comments and suggestions. 
	
	\section{Monotonicity property}\label{S2}
	
	We begin this section with an observation that leads to the first part of Theorem~\ref{T1}.
	
	\begin{proposition}\label{P1}
		Fix a pair $(\XXX, \BBB)$ and $p_1,p_2 \in [1,\infty)$ such that $p_1 < p_2$. Then
		\[
		p_1 \in {\rm diff}(\BBB) \implies p_2 \in {\rm diff}
		(\BBB)
		\]
		if any of the following two cases occurs:
		\begin{enumerate} [label={\rm (\alph*)}]
			\item $\XXX$ is complete and $p_2 < \infty$,
			\item $X = \bigcup_{n \in \NN} G_n$ for open sets $G_n \subseteq X$ satisfying $\mu(G_n) < \infty$. 
		\end{enumerate}
	\end{proposition}
	
	\begin{proof}
		First, assume that $\XXX$ is complete and $p_2 < \infty$. Fix $f \in L^{p_2}(\XXX)$. We shall show that $E(f)$ is a set of measure zero. For every $n \in \NN$, define $f_n \coloneqq f \cdot \ind{F_n}$, where 
		\[
		F_n \coloneqq \{x \in X : |f(x)| \geq 2^{-n} \}.
		\]
		Note that $\|f-f_n\|_{L^\infty(\XXX)} \leq 2^{-n}$ and $f_n \in L^{p_1}(\XXX)$ because
		\[
		\| f_n \|_{L^{p_1}(\XXX)}^{p_1} =
		\int_X |f_n|^{p_1} \, {\rm d}\mu(x) \leq
		\int_X |f|^{p_1} \frac{|f|^{p_2-p_1}}{2^{-n(p_2-p_1)}} \, {\rm d}\mu(x) 
		= 2^{n(p_2 - p_1)} \| f \|_{L^{p_2}(\XXX)}^{p_2} < \infty.
		\]
		We also deduce that
		\[
		\|\overline{f}-\overline{f_n}\|_{L^\infty(\XXX)} \leq 2^{-n}
		\quad \text{and} \quad
		\|\underline{f}-\underline{f_n}\|_{L^\infty(\XXX)} \leq 2^{-n}.
		\]
		Now, suppose that $x \in E(f)$. Then, for some $n \in \NN$, we have 
		\[
		|\overline{f}(x) - f(x)| > 2^{-n+1}
		\quad \text{or} \quad 
		|\underline{f}(x) - f(x)| > 2^{-n+1}.
		\]
		We claim that if one of these inequalities holds, then $x \in E(f_n)$. Indeed, observe that
		\[
		|\overline{f_n}(x) - f_n(x)| \geq
		|\overline{f}(x) - f(x)| 
		- \|\overline{f}-\overline{f_n}\|_{L^\infty(\XXX)}
		- \|f-f_n\|_{L^\infty(\XXX)} > 0
		\]
		if the first inequality holds and, analogously, $|\underline{f_n}(x) - f_n(x)| > 0$ if the second inequality holds. 
		This proves our claim and, consequently, we obtain $E(f) \subseteq \bigcup_{n \in \NN} E(f_n)$. Since $p_1 \in {\rm diff}(\BBB)$, we have $\mu(E(f_n)) = 0$ for every $n \in \NN$. Therefore, $\bigcup_{n \in \NN} E(f_n)$ is a set of measure zero which implies $\mu(E(f)) = 0$ because $\XXX$ is complete. 
		
		Next, assume that $X = \bigcup_{n \in \NN} G_n$ for open sets $G_n \subseteq X$ satisfying $\mu(G_n) < \infty$. As before, fix $f \in L^{p_2}(\XXX)$ and, for every $n \in \NN$, define $f_n \coloneqq f \cdot \ind{F_n}$, where
		\[
		F_n \coloneqq G_1 \cup \cdots \cup G_n.
		\]
		Since $\mu(F_n) < \infty$, we have $f_n \in L^{p_1}(\XXX)$ so that $\mu(E(f_n)) = 0$ because $p_1 \in {\rm diff}(\BBB)$. Moreover, since $F_n$ is open and $f_n$ agrees with $f$ on $F_n$, we have $E(f) \cap F_n = E(f_n) \cap F_n$. Hence,
		\[
		E(f) = \bigcup_{n \in \NN} \big( E(f) \cap F_n \big) = \bigcup_{n \in \NN} \big( E(f_n) \cap F_n \big)
		\]
		and the set on the right-hand side has measure zero, as desired. 
	\end{proof}
	
	\begin{remark} \label{rem1}
		The only instance where we use the completeness of $\XXX$ is in deducing that $E(f)$ is measurable from the inclusion $E(f) \subseteq \bigcup_{n \in \NN} E(f_n)$ and the fact that $\bigcup_{n \in \NN} E(f_n)$ has measure zero. The measurability aspect cannot be ignored, as in general $E(f)$ may not be measurable (see Example~\ref{E2} and Example~\ref{E3}). Moreover, $E(f) \subsetneq \bigcap_{n \in \NN} E(f_n)$ may even happen (see Example~\ref{E4}). Nonetheless, it remains unclear whether completeness is truly essential; in particular, it is not whenever we know \emph{a~priori} that $E(f)$ is measurable for every $p \in [1,\infty]$ and $f \in L^p(\XXX)$.  
	\end{remark}
	
	The example below is key to understanding why the case $p_2=\infty$ in Proposition~\ref{P1} is special. Namely, we construct a pair $(\XXX,\BBB)$ such that ${\rm diff}(\BBB) = [1,\infty)$.
	
	\begin{example} \label{E1}
		Let $\XXX$ be defined as follows (see Figure~\ref{Pic0}). Set $X \coloneqq I \cup K \subseteq \RR^2$, where
		\begin{align*}
			I & \coloneqq \{(u,0) : u \in (0,1]\}, \\
			K & \coloneqq \{ (u, 2^{-j}) : u \in (0,1], \, j \in \NN\}. 
		\end{align*}
		For $(u,v), (u',v') \in X$, define 
		\[
		\rho((u,v),(u',v')) \coloneqq 
		\begin{cases}
		|v-v'| & \text{if } u=u', \\
		1 & \text{if } u \neq u'. 
		\end{cases} 
		\]
		Finally, set $\mu \coloneqq \mu_I + \mu_K$, where $\mu_I(E) \coloneqq \infty$ if $E \cap I$ is uncountable or $\mu_I(E) \coloneqq 0$ otherwise, while $\mu_K$ is the counting measure on $K$. Notice that $\XXX$ is complete because all subsets of $X$ are measurable. 
		
		Next, let $\BBB$ be an uncentered basis defined as follows. Set $\BBB \coloneqq \BBB_K \cup \BBB_{IK}$, where
		\begin{align*}
			\BBB_K & \coloneqq \big\{ \{(u, 2^{-j})\} : u \in (0,1], \, j \in \NN \big\}, \\
			\BBB_{IK} & \coloneqq \big\{ \{(u,0),(u, 2^{-j})\} : u \in (0,1], \, j \in \NN \big\}.
		\end{align*}
		
		Now, fix $p \in [1,\infty)$ and let $f \in L^p(\XXX)$. We shall verify that condition~\eqref{D} holds for $\mu$-almost every $x \in X$. First, for every $x \in K$, condition~\eqref{D} holds trivially because if $(S_n)_{n \in \NN} \Rightarrow x$, then $S_n = \{x\}$ for all sufficiently large $n \in \NN$. On the other hand, for every $x \in I$, if $(S_n)_{n \in \NN} \Rightarrow x$, then $S_n \in \BBB_{IK}$ for all sufficiently large $n \in \NN$. Since $f \in L^p(\XXX)$, we have $\lim_{n \in \NN} {\rm Avg}_f(S_n) = 0$. Moreover, $f(x)=0$ for $\mu$-almost every $x \in I$, again due to the fact that $f \in L^p(\XXX)$. Hence, condition~\eqref{D} holds for $\mu$-almost every $x \in I$.
		
		Finally, let $p=\infty$ and consider $g \coloneqq \ind{K} \in L^{\infty}(\XXX)$. It is easy to verify that $\overline{g} = \underline{g} = \ind{X}$. Hence, $\mu(E(g)) = \mu(I) = \infty$ and so $\BBB$ fails to differentiate  $L^{\infty}(\XXX)$.  
	\end{example}
	
	\begin{center}
		\begin{figure}[H]
			\begin{tikzpicture}
			[
			scale=1.75,
			axis/.style={very thick,->,>=stealth'},
			important line/.style={thick},
			dashed line/.style={dashed, thin}
			]
			
			\draw[axis] (-0.75,0) -- (4.75,0) node(xline)[right]
			{};
			\draw[axis] (0,-0.75) -- (0,2.75) node(yline)[above] {};
			
			\draw (4,-0.1) node[below] {$1$} -- (4,0.1);
			\draw (-0.1,2) node[left] {$\frac{1}{2}$} -- (0.1,2);
			\draw (-0.1,1) node[left] {$\frac{1}{4}$} -- (0.1,1);
			\draw[important line, line width=1mm] (0,0) -- (4,0);
			
			\draw[dashed line, line width=0.1mm] (0.25,0) -- (0.25,2);
			\draw[dashed line, line width=0.1mm] (0.5,0) -- (0.5,2);
			\draw[dashed line, line width=0.1mm] (0.75,0) -- (0.75,2);
			\draw[dashed line, line width=0.1mm] (1,0) -- (1,2);
			\draw[dashed line, line width=0.1mm] (1.25,0) -- (1.25,2);
			\draw[dashed line, line width=0.1mm] (1.5,0) -- (1.5,2);
			\draw[dashed line, line width=0.1mm] (1.75,0) -- (1.75,2);
			\draw[dashed line, line width=0.1mm] (2,0) -- (2,2);
			\draw[dashed line, line width=0.1mm] (2.25,0) -- (2.25,2);
			\draw[dashed line, line width=0.1mm] (2.5,0) -- (2.5,2);
			\draw[dashed line, line width=0.1mm] (2.75,0) -- (2.75,2);
			\draw[dashed line, line width=0.1mm] (3,0) -- (3,2);
			\draw[dashed line, line width=0.1mm] (3.25,0) -- (3.25,2);
			\draw[dashed line, line width=0.1mm] (3.5,0) -- (3.5,2);
			\draw[dashed line, line width=0.1mm] (3.75,0) -- (3.75,2);
			\draw[dashed line, line width=0.1mm] (4,0) -- (4,2);
			
			\node at (0.25,2)[circle,fill,inner sep=1.5pt]{};
			\node at (0.5,2)[circle,fill,inner sep=1.5pt]{};
			\node at (0.75,2)[circle,fill,inner sep=1.5pt]{};
			\node at (1,2)[circle,fill,inner sep=1.5pt]{};
			\node at (1.25,2)[circle,fill,inner sep=1.5pt]{};
			\node at (1.5,2)[circle,fill,inner sep=1.5pt]{};
			\node at (1.75,2)[circle,fill,inner sep=1.5pt]{};
			\node at (2,2)[circle,fill,inner sep=1.5pt]{};
			\node at (2.25,2)[circle,fill,inner sep=1.5pt]{};
			\node at (2.5,2)[circle,fill,inner sep=1.5pt]{};
			\node at (2.75,2)[circle,fill,inner sep=1.5pt]{};
			\node at (3,2)[circle,fill,inner sep=1.5pt]{};
			\node at (3.25,2)[circle,fill,inner sep=1.5pt]{};
			\node at (3.5,2)[circle,fill,inner sep=1.5pt]{};
			\node at (3.75,2)[circle,fill,inner sep=1.5pt]{};
			\node at (4,2)[circle,fill,inner sep=1.5pt]{};
			
			\node at (0.25,1)[circle,fill,inner sep=1.5pt]{};
			\node at (0.5,1)[circle,fill,inner sep=1.5pt]{};
			\node at (0.75,1)[circle,fill,inner sep=1.5pt]{};
			\node at (1,1)[circle,fill,inner sep=1.5pt]{};
			\node at (1.25,1)[circle,fill,inner sep=1.5pt]{};
			\node at (1.5,1)[circle,fill,inner sep=1.5pt]{};
			\node at (1.75,1)[circle,fill,inner sep=1.5pt]{};
			\node at (2,1)[circle,fill,inner sep=1.5pt]{};
			\node at (2.25,1)[circle,fill,inner sep=1.5pt]{};
			\node at (2.5,1)[circle,fill,inner sep=1.5pt]{};
			\node at (2.75,1)[circle,fill,inner sep=1.5pt]{};
			\node at (3,1)[circle,fill,inner sep=1.5pt]{};
			\node at (3.25,1)[circle,fill,inner sep=1.5pt]{};
			\node at (3.5,1)[circle,fill,inner sep=1.5pt]{};
			\node at (3.75,1)[circle,fill,inner sep=1.5pt]{};
			\node at (4,1)[circle,fill,inner sep=1.5pt]{};
			
			\node at (0.25,0.5)[circle,fill,inner sep=1.5pt]{};
			\node at (0.5,0.5)[circle,fill,inner sep=1.5pt]{};
			\node at (0.75,0.5)[circle,fill,inner sep=1.5pt]{};
			\node at (1,0.5)[circle,fill,inner sep=1.5pt]{};
			\node at (1.25,0.5)[circle,fill,inner sep=1.5pt]{};
			\node at (1.5,0.5)[circle,fill,inner sep=1.5pt]{};
			\node at (1.75,0.5)[circle,fill,inner sep=1.5pt]{};
			\node at (2,0.5)[circle,fill,inner sep=1.5pt]{};
			\node at (2.25,0.5)[circle,fill,inner sep=1.5pt]{};
			\node at (2.5,0.5)[circle,fill,inner sep=1.5pt]{};
			\node at (2.75,0.5)[circle,fill,inner sep=1.5pt]{};
			\node at (3,0.5)[circle,fill,inner sep=1.5pt]{};
			\node at (3.25,0.5)[circle,fill,inner sep=1.5pt]{};
			\node at (3.5,0.5)[circle,fill,inner sep=1.5pt]{};
			\node at (3.75,0.5)[circle,fill,inner sep=1.5pt]{};
			\node at (4,0.5)[circle,fill,inner sep=1.5pt]{};
			
			\node at (0.25,0.25)[circle,fill,inner sep=1.5pt]{};
			\node at (0.5,0.25)[circle,fill,inner sep=1.5pt]{};
			\node at (0.75,0.25)[circle,fill,inner sep=1.5pt]{};
			\node at (1,0.25)[circle,fill,inner sep=1.5pt]{};
			\node at (1.25,0.25)[circle,fill,inner sep=1.5pt]{};
			\node at (1.5,0.25)[circle,fill,inner sep=1.5pt]{};
			\node at (1.75,0.25)[circle,fill,inner sep=1.5pt]{};
			\node at (2,0.25)[circle,fill,inner sep=1.5pt]{};
			\node at (2.25,0.25)[circle,fill,inner sep=1.5pt]{};
			\node at (2.5,0.25)[circle,fill,inner sep=1.5pt]{};
			\node at (2.75,0.25)[circle,fill,inner sep=1.5pt]{};
			\node at (3,0.25)[circle,fill,inner sep=1.5pt]{};
			\node at (3.25,0.25)[circle,fill,inner sep=1.5pt]{};
			\node at (3.5,0.25)[circle,fill,inner sep=1.5pt]{};
			\node at (3.75,0.25)[circle,fill,inner sep=1.5pt]{};
			\node at (4,0.25)[circle,fill,inner sep=1.5pt]{};
			
			\node (A) at (0,0) {};
			\node (E) at (4,0) {};
			\node (I) at (4,2) {};
			
			\draw [thick,
			decoration={brace,mirror,raise=0.15cm},
			decorate] (A) -- (E)
			node [pos=0.5,anchor=north,yshift=-0.2cm] {$I$};
			
			\draw [thick,
			decoration={brace,mirror,raise=0.15cm},
			decorate] (E) -- (I)
			node [pos=0.5,anchor=north,xshift=0.5cm, yshift=0.3cm] {$K$};
			
			\end{tikzpicture}
			\vspace{-1em}
			\caption{The space $\XXX$ from Example~\ref{E1}.}
			\label{Pic0}
			\vspace{-1.5em}
		\end{figure}
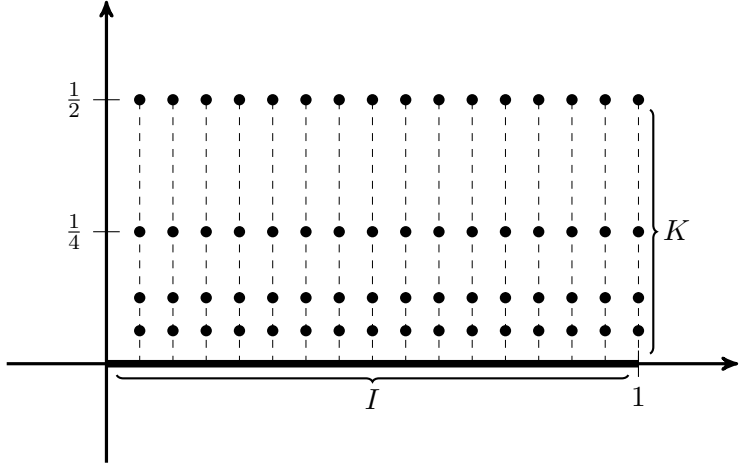
	\end{center}
	
	Example~\ref{E1} is quite peculiar. Namely, we have $E(\ind{K}) = I$, while there is no subset $I' \subseteq I$ such that $\mu(I') \in (0, \infty)$. We comment on this issue below. 
	
	\begin{remark} \label{rem2}
		Let $(\XXX, \BBB)$ be a pair such that $\infty \notin {\rm diff}(\BBB)$. We claim that either ${\rm diff}(\BBB) = \emptyset$ or, for every $f \in L^\infty(\XXX)$, the subsets of $E(f)$ satisfy a certain zero--infinity dichotomy. 
		
		To see this, consider $g \in L^\infty(\XXX)$ such that $E(g)$  is not a set of measure zero. Suppose that there exists a subset $E \subseteq E(g)$ which has a superset $E'$ of finite measure but has no supersets of measure zero (including $E$ itself). Fix $p \in [1,\infty)$. Since
		\[
		E = \{ x \in E : \overline{g}(x) > g(x) \} \cup  \{ x \in E :  \underline{g}(x) < g(x)\},
		\]
		one of the sets on the right-hand side, call them $E_1$ and $E_2$, has no supersets of measure zero. Suppose that $E_2$ has this property (the remaining case is symmetric). We define
		\[
		h \coloneqq (g + \|g\|_{L^\infty(\XXX)}) \cdot \ind{E'} \in L^p(\XXX). 
		\]
		Since $g + \|g\|_{L^\infty(\XXX)}$ is nonnegative, we obtain $E_2 \subseteq \{x \in E : \underline{h}(x) < h(x)\} \subseteq E(h)$. It follows that $E(h)$, as a superset of $E_2$, is not a set of measure zero. Hence, $p \notin {\rm diff}(\BBB)$.
		
		To summarize, if $\infty \notin {\rm diff}(\BBB)$ and ${\rm diff}(\BBB) \neq \emptyset$, then for every $f \in L^\infty(\XXX)$ and every $E \subseteq E(f)$ either $E$ has a superset of measure zero or $E$ has no supersets of finite measure.
	\end{remark}
	
	Next, we present two examples illustrating the measurability issues. While the first is somewhat more natural, the second involves a differentiation basis which is uncentered.
	
	\begin{example} \label{E2}
		Let $\XXX$ be defined as follows (see Figure~\ref{Pic1}). Set $X \coloneqq I \cup K \subseteq \RR^2$, where
		\begin{align*}
			I & \coloneqq \{(u,0) : u \in (0,1]\}, \\
			K & \coloneqq \{ (i 2^{-j}, 2^{-j}) : (i,j) \in \Delta\}.
		\end{align*}
		Here $\Delta \coloneqq \{(i,j) \in \NN^2 : i \leq 2^j\} \subseteq \NN^2$. Let $\rho$ be the usual two-dimensional Euclidean distance function truncated to $X \times X$. Finally, set $\mu \coloneqq \mu_I + \mu_K$, where $\mu_I$ is the one-dimensional Lebesgue measure on $I$ and $\mu_K$ is the counting measure on $K$. 
		
		Next, let $\BBB$ be defined as follows. Set $\BBB \coloneqq \BBB_I \cup \BBB_K \cup \BBB_{IK}$, where
		\begin{align*}
			\BBB_I & \coloneqq \{ I_{ij} : (i,j) \in \Delta \}, \\
			\BBB_K & \coloneqq \{ K_{ij} : (i,j) \in \Delta \}, \\
			\BBB_{IK} & \coloneqq \{ I_{ij} \cup K_{ij} : (i,j) \in \Delta \}.
		\end{align*}
		Here, for every $(i,j) \in \Delta$, we use the auxiliary sets
		\begin{align*}
			I_{ij} & \coloneqq \{(u,0) : u \in ((i-1)2^{-j}, i2^{-j}]\}, \\
			K_{ij} & \coloneqq \{(i 2^{-j}, 2^{-j})\}.
		\end{align*}
		This time we do not consider $\BBB$ as an uncentered basis. Instead, we choose any subset $F \subseteq (0,1]$ which is not measurable with respect to Lebesgue measure on $\RR$, and set
		\[
		\BBB(x) \coloneqq
		\begin{cases}
		\{E \in \BBB\phantom{_I} : x \in E\} & \text{if } x \in K \cup (F \times \{0\}), \\
		\{E \in \BBB_I : x \in E\} & \text{if } x \in I \setminus (F \times \{0\}).
		\end{cases}
		\]
		
		Fix $p \in [1,\infty]$ and consider $g \coloneqq \ind{I} \in L^p(\XXX)$. Then $\overline{g} = \ind{I}$, since a point $x \in X$ can be approximated by the sets $I_{i,j}$ if and only if $x \in I$. Moreover, if $x \in I \setminus F$, then $x$ cannot be approximated by the sets $I_{ij} \cup K_{ij}$ so that $\underline{g} = \ind{I \setminus F}$. Hence, $E(g)$ is not measurable.   
	\end{example}
	
	\begin{example} \label{E3}
		Let $\XXX$ be as in Example~\ref{E2}. Next, let $\BBB$ be an uncentered basis defined as follows. Set $\BBB \coloneqq \BBB_I \cup \BBB_K \cup \BBB_{IK}$, where $\BBB_I$ and $\BBB_K$ are as in Example~\ref{E2}, while
		\[
		\BBB_{IK} \coloneqq \big\{ \{x\} \cup K_{ij} : x \in I_{ij} \cap F, \, (i,j) \in \Delta \big\}.
		\]
		Here $\Delta, I_{ij}, K_{ij}, F$ are also as in Example~\ref{E2}.
		As before, fix $p \in [1,\infty]$ and consider $g \coloneqq \ind{I} \in L^p(\XXX)$. Then $\overline{g} = \ind{I}$ and $\underline{g} = \ind{I \setminus F}$. Hence, $E(g)$ is not measurable. The relevance of Example~\ref{E3} will be understood after the discussion at the beginning of Section~\ref{S3}.
	\end{example}
	
	\begin{center}
		\begin{figure}[H]
			\begin{tikzpicture}
			[
			scale=1.75,
			axis/.style={very thick,->,>=stealth'},
			important line/.style={thick},
			dashed line/.style={dashed, thin}
			]
			
			\draw[axis] (-0.75,0) -- (4.75,0) node(xline)[right]
			{};
			\draw[axis] (0,-0.75) -- (0,2.75) node(yline)[above] {};
			
			\draw (4,-0.1) node[below] {$1$} -- (4,0.1);
			\draw (2,-0.1) node[below] {$\frac{1}{2}$} -- (2,0.1);
			\draw (-0.1,2) node[left] {$\frac{1}{2}$} -- (0.1,2);
			\draw (-0.1,1) node[left] {$\frac{1}{4}$} -- (0.1,1);
			\draw[important line, line width=1mm] (0,0) -- (4,0);
			
			\draw[dashed line, line width=0.1mm] (4,0) -- (4,2);
			\draw[dashed line, line width=0.1mm] (2,0) -- (2,2);
			\draw[dashed line, line width=0.1mm] (1,0) -- (1,1);
			\draw[dashed line, line width=0.1mm] (3,0) -- (3,1);
			
			\draw[dashed line, line width=0.1mm] (0.5,0) -- (0.5,0.5);
			\draw[dashed line, line width=0.1mm] (1.5,0) -- (1.5,0.5);
			\draw[dashed line, line width=0.1mm] (2.5,0) -- (2.5,0.5);
			\draw[dashed line, line width=0.1mm] (3.5,0) -- (3.5,0.5);
			
			\draw[dashed line, line width=0.1mm] (2,0) -- (4,2);
			\draw[dashed line, line width=0.1mm] (0,0) -- (2,2);
			\draw[dashed line, line width=0.1mm] (1,0) -- (2,1);
			\draw[dashed line, line width=0.1mm] (3,0) -- (4,1);
			
			\draw[dashed line, line width=0.1mm] (0.5,0) -- (1,0.5);
			\draw[dashed line, line width=0.1mm] (1.5,0) -- (2,0.5);
			\draw[dashed line, line width=0.1mm] (2.5,0) -- (3,0.5);
			\draw[dashed line, line width=0.1mm] (3.5,0) -- (4,0.5);
			
			\draw[dashed line, line width=0.1mm] (0.25,0) -- (0.25,0.25);
			\draw[dashed line, line width=0.1mm] (1.25,0) -- (1.25,0.25);
			\draw[dashed line, line width=0.1mm] (2.25,0) -- (2.25,0.25);
			\draw[dashed line, line width=0.1mm] (3.25,0) -- (3.25,0.25);
			\draw[dashed line, line width=0.1mm] (0.75,0) -- (0.75,0.25);
			\draw[dashed line, line width=0.1mm] (1.75,0) -- (1.75,0.25);
			\draw[dashed line, line width=0.1mm] (2.75,0) -- (2.75,0.25);
			\draw[dashed line, line width=0.1mm] (3.75,0) -- (3.75,0.25);
			
			\draw[dashed line, line width=0.1mm] (0.25,0) -- (0.5,0.25);
			\draw[dashed line, line width=0.1mm] (1.25,0) -- (1.5,0.25);
			\draw[dashed line, line width=0.1mm] (2.25,0) -- (2.5,0.25);
			\draw[dashed line, line width=0.1mm] (3.25,0) -- (3.5,0.25);
			\draw[dashed line, line width=0.1mm] (0.75,0) -- (1,0.25);
			\draw[dashed line, line width=0.1mm] (1.75,0) -- (2,0.25);
			\draw[dashed line, line width=0.1mm] (2.75,0) -- (3,0.25);
			\draw[dashed line, line width=0.1mm] (3.75,0) -- (4,0.25);
			
			\node at (2,2)[circle,fill,inner sep=1.5pt]{};
			\draw (2,2) node[right] {$K_{11}$} -- (2,2);
			\node at (4,2)[circle,fill,inner sep=1.5pt]{};
			\draw (4,2) node[right] {$K_{21}$} -- (4,2);
			
			\node at (1,1)[circle,fill,inner sep=1.5pt]{};
			\draw (1,1) node[right] {$K_{12}$} -- (1,1);
			\node at (2,1)[circle,fill,inner sep=1.5pt]{};
			\draw (2,1) node[right] {$K_{22}$} -- (2,1);
			\node at (3,1)[circle,fill,inner sep=1.5pt]{};
			\draw (3,1) node[right] {$K_{32}$} -- (3,1);
			\node at (4,1)[circle,fill,inner sep=1.5pt]{};
			\draw (4,1) node[right] {$K_{42}$} -- (4,1);
			
			\node at (0.5,0.5)[circle,fill,inner sep=1.5pt]{};
			\node at (1,0.5)[circle,fill,inner sep=1.5pt]{};
			\node at (1.5,0.5)[circle,fill,inner sep=1.5pt]{};
			\node at (2,0.5)[circle,fill,inner sep=1.5pt]{};
			\node at (2.5,0.5)[circle,fill,inner sep=1.5pt]{};
			\node at (3,0.5)[circle,fill,inner sep=1.5pt]{};
			\node at (3.5,0.5)[circle,fill,inner sep=1.5pt]{};
			\node at (4,0.5)[circle,fill,inner sep=1.5pt]{};
			
			\node at (0.25,0.25)[circle,fill,inner sep=1.5pt]{};
			\node at (0.5,0.25)[circle,fill,inner sep=1.5pt]{};
			\node at (0.75,0.25)[circle,fill,inner sep=1.5pt]{};
			\node at (1,0.25)[circle,fill,inner sep=1.5pt]{};
			\node at (1.25,0.25)[circle,fill,inner sep=1.5pt]{};
			\node at (1.5,0.25)[circle,fill,inner sep=1.5pt]{};
			\node at (1.75,0.25)[circle,fill,inner sep=1.5pt]{};
			\node at (2,0.25)[circle,fill,inner sep=1.5pt]{};
			\node at (2.25,0.25)[circle,fill,inner sep=1.5pt]{};
			\node at (2.5,0.25)[circle,fill,inner sep=1.5pt]{};
			\node at (2.75,0.25)[circle,fill,inner sep=1.5pt]{};
			\node at (3,0.25)[circle,fill,inner sep=1.5pt]{};
			\node at (3.25,0.25)[circle,fill,inner sep=1.5pt]{};
			\node at (3.5,0.25)[circle,fill,inner sep=1.5pt]{};
			\node at (3.75,0.25)[circle,fill,inner sep=1.5pt]{};
			\node at (4,0.25)[circle,fill,inner sep=1.5pt]{};
			
			\node (A) at (0,0) {};
			\node (B) at (1,0) {};
			\node (C) at (2,0) {};
			\node (D) at (3,0) {};
			\node (E) at (4,0) {};
			
			\draw [thick,
			decoration={brace,mirror,raise=0.15cm},
			decorate] (A) -- (C)
			node [pos=0.5,anchor=north,yshift=-0.2cm] {$I_{11}$};
			\draw [thick,
			decoration={brace,mirror,raise=0.15cm},
			decorate] (C) -- (E)
			node [pos=0.5,anchor=north,yshift=-0.2cm] {$I_{21}$};
			
			\draw [thick,
			decoration={brace,mirror,raise=0.8cm},
			decorate] (A) -- (B)
			node [pos=0.5,anchor=north,yshift=-0.85cm] {$I_{12}$};
			\draw [thick,
			decoration={brace,mirror,raise=0.8cm},
			decorate] (B) -- (C)
			node [pos=0.5,anchor=north,yshift=-0.85cm] {$I_{22}$};
			\draw [thick,
			decoration={brace,mirror,raise=0.8cm},
			decorate] (C) -- (D)
			node [pos=0.5,anchor=north,yshift=-0.85cm] {$I_{32}$};
			\draw [thick,
			decoration={brace,mirror,raise=0.8cm},
			decorate] (D) -- (E)
			node [pos=0.5,anchor=north,yshift=-0.85cm] {$I_{42}$};
			
			\end{tikzpicture}
			\vspace{-1em}
			\caption{The space $\XXX$ from Example~\ref{E2} and Example~\ref{E3}.}
			\label{Pic1}
			\vspace{-1.5em}
		\end{figure}
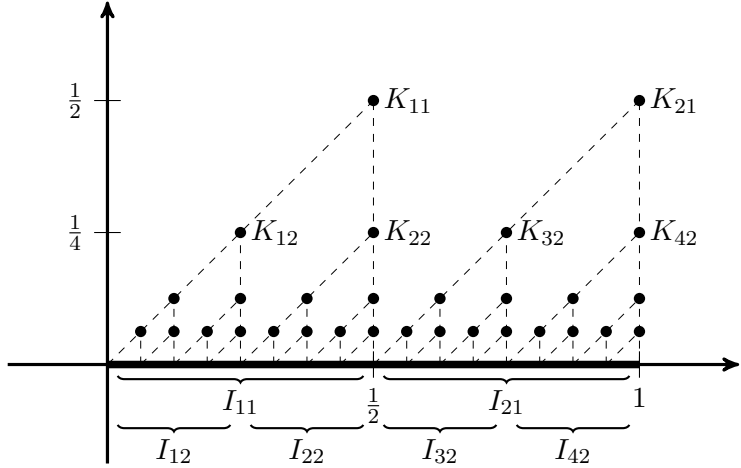
	\end{center}
	
	We conclude this section with an example showing that $E(f) \subsetneq \bigcap_{n \in \NN} E(f_n)$ may happen for some $f$ and $f_n$ as in the first part of the proof of Proposition~\ref{P1} (cf.~Remark~\ref{rem1}).
	
	\begin{example} \label{E4}
		Let $\XXX$ be defined as follows (see Figure~\ref{Pic2}). Set $X \coloneqq I \cup K \subseteq \RR^2$, where
		\begin{align*}
			I & \coloneqq \{(u,0) : u \in (0,1]\}, \\
			K & \coloneqq \{ (i 2^{-2j}, 2^{-j}) : (i,j) \in \nabla \}.
		\end{align*}
		Here $\nabla \coloneqq \{(i,j) \in \NN^2 : i \leq 2^{2j}\} \subseteq \NN^2$. Let $\rho$ be the usual two-dimensional Euclidean distance function truncated to $X \times X$. Finally, set $\mu \coloneqq \mu_I + \mu_K$, where $\mu_I$ is the one-dimensional Lebesgue measure on $I$ and $\mu_K$ is a measure supported on $K$ such that
		\[
		\mu_K(\{ (i 2^{-2j}, 2^{-j}) \}) = 2^{-2j - r(i,j)}
		\]
		for every $(i,j) \in \nabla$, where $r(i,j) \coloneqq 2^j \lceil i/2^j \rceil  - i$. Observe that $\mu(X) < \infty$. 
		
		Let $\BBB$ be an uncentered basis defined as follows. Set $\BBB \coloneqq \BBB_I \cup \BBB_K \cup \BBB_{IK}$, where
		\begin{align*}
			\BBB_I & \coloneqq \{ I_{ij} : (i,j) \in \nabla \}, \\
			\BBB_K & \coloneqq \{ K_{ij} : (i,j) \in \nabla \}, \\
			\BBB_{IK} & \coloneqq \{ I_{ij} \cup K^*_{ij} : (i,j) \in \nabla \}.
		\end{align*}
		Here, for every $(i,j) \in \nabla$, we use the auxiliary sets
		\begin{align*}
			I_{ij} & \coloneqq \{(u,0) : u \in ((i-1)2^{-2j}, i2^{-2j}]\}, \\
			K_{ij} & \coloneqq \{(i 2^{-2j}, 2^{-j})\},\\
			K^*_{ij} & \coloneqq \{(k 2^{-2j}, 2^{-j}) : k \in \NN  \text{ such that } \lceil k/2^j \rceil = \lceil i/2^j \rceil \}.
		\end{align*}
		
		Consider $g = g_I + g_K$, where $g_I \coloneqq \frac{2}{3} \ind{I}$ and $g_K$ is a function supported on $K$ such that
		\[
		g_K((i 2^{-2j}, 2^{-j})) \coloneqq 2^{- r(i,j)}
		\]
		for every $(i,j) \in \nabla$. 
		Then $g \in L^p(\XXX)$ for any $p \in [1,\infty]$, and $\overline{g} = \underline{g} = g$ because, for every $x \in I$, approximating $x$ by sets of the form $I_{ij}$ or $I_{ij} \cup K_{ij}^*$ leads to the same limiting value 
		\[
		\lim_{j \to \infty} \frac{2^{-2j} \cdot 2/3}{2^{-2j}}
		= \frac{2}{3} 
		= \lim_{j \to \infty} \frac{2^{-2j} \cdot 2/3 + 2^{-2j}(1 + 1/4 + \cdots + 1/2^{2j})}{2^{-2j} + 2^{-2j}(1 + 1/2 + \cdots + 1/2^j)}.
		\]
		On the other hand, for every $n \in \NN$ and $g_n \coloneqq g \cdot \ind{F_n}$, where $F_n \coloneqq \{x \in X : |g(x)| \geq 2^{-n} \}$, we obtain $\overline{g_n}(x) = \underline{g_n}(x) = g_n(x)$ for every $x \in K$, while for every $x \in I$ we have 
		\[
		\underline{g_n}(x) = 
		\lim_{j \to \infty} \frac{2^{-2j} \cdot 2/3 + 2^{-2j}(1 + 1/4 + \cdots + 1/2^{2n})}{2^{-2j} + 2^{-2j}(1 + 1/2 + \cdots + 1/2^j)} < \frac{2}{3} = g_n(x) = \overline{g_n}(x),
		\]
		since approximating $x$ by sets of the form $I_{ij} \cup K_{ij}^*$ leads to the smaller value above. Thus, $E(g) = \emptyset$ but $E(g_n) = I$ for every $n \in \NN$ so that $E(g) \subsetneq \bigcap_{n \in \NN} E(g_n)$, as claimed. 
	\end{example}
	
	\begin{center}
		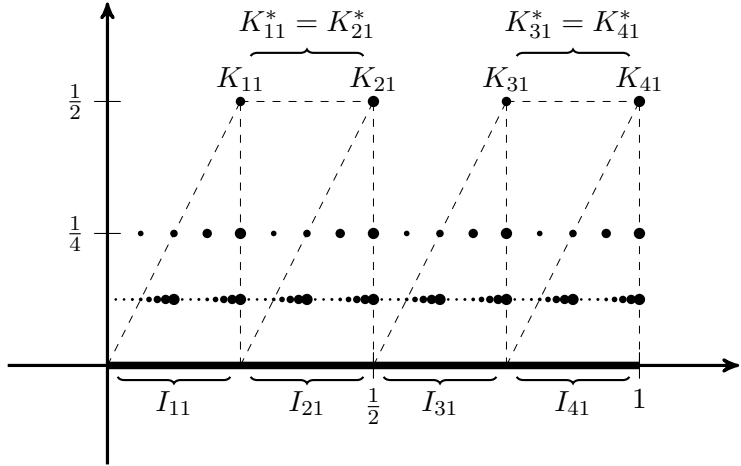
\begin{figure}[H]
			\begin{tikzpicture}
			[
			scale=1.75,
			axis/.style={very thick,->,>=stealth'},
			important line/.style={thick},
			dashed line/.style={dashed, thin}
			]
			
			\draw[axis] (-0.75,0) -- (4.75,0) node(xline)[right]
			{};
			\draw[axis] (0,-0.75) -- (0,2.75) node(yline)[above] {};
			
			\draw (4,-0.1) node[below] {$1$} -- (4,0.1);
			\draw (2,-0.1) node[below] {$\frac{1}{2}$} -- (2,0.1);
			\draw (-0.1,2) node[left] {$\frac{1}{2}$} -- (0.1,2);
			\draw (-0.1,1) node[left] {$\frac{1}{4}$} -- (0.1,1);
			\draw[important line, line width=1mm] (0,0) -- (4,0);
			
			\draw[dashed line, line width=0.1mm] (4,0) -- (4,2);
			\draw[dashed line, line width=0.1mm] (3,0) -- (3,2);
			\draw[dashed line, line width=0.1mm] (3,0) -- (4,2);
			\draw[dashed line, line width=0.1mm] (2,0) -- (3,2);
			
			\draw[dashed line, line width=0.1mm] (2,0) -- (2,2);
			\draw[dashed line, line width=0.1mm] (1,0) -- (1,2);
			\draw[dashed line, line width=0.1mm] (1,0) -- (2,2);
			\draw[dashed line, line width=0.1mm] (0,0) -- (1,2);
			
			\draw[dashed line, line width=0.1mm] (3,2) -- (4,2);
			\draw[dashed line, line width=0.1mm] (1,2) -- (2,2);
			
			\node at (1,2)[circle,fill,inner sep=1.25pt]{};
			\draw (1,2) node[above] {$K_{11}$} -- (1,2);
			\node at (2,2)[circle,fill,inner sep=1.5pt]{};
			\draw (2,2) node[above] {$K_{21}$} -- (2,2);
			\node at (3,2)[circle,fill,inner sep=1.25pt]{};
			\draw (3,2) node[above] {$K_{31}$} -- (3,2);
			\node at (4,2)[circle,fill,inner sep=1.5pt]{};
			\draw (4,2) node[above] {$K_{41}$} -- (4,2);
			
			\node at (0.25,1)[circle,fill,inner sep=0.75pt]{};
			\node at (0.5,1)[circle,fill,inner sep=1pt]{};
			\node at (0.75,1)[circle,fill,inner sep=1.25pt]{};
			\node at (1,1)[circle,fill,inner sep=1.5pt]{};
			\node at (1.25,1)[circle,fill,inner sep=0.75pt]{};
			\node at (1.5,1)[circle,fill,inner sep=1pt]{};
			\node at (1.75,1)[circle,fill,inner sep=1.25pt]{};
			\node at (2,1)[circle,fill,inner sep=1.5pt]{};
			\node at (2.25,1)[circle,fill,inner sep=0.75pt]{};
			\node at (2.5,1)[circle,fill,inner sep=1pt]{};
			\node at (2.75,1)[circle,fill,inner sep=1.25pt]{};
			\node at (3,1)[circle,fill,inner sep=1.5pt]{};
			\node at (3.25,1)[circle,fill,inner sep=0.75pt]{};
			\node at (3.5,1)[circle,fill,inner sep=1pt]{};
			\node at (3.75,1)[circle,fill,inner sep=1.25pt]{};
			\node at (4,1)[circle,fill,inner sep=1.5pt]{};
			
			\node at (0.0625,0.5)[circle,fill,inner sep=0.2pt]{};
			\node at (0.125,0.5)[circle,fill,inner sep=0.3pt]{};
			\node at (0.1875,0.5)[circle,fill,inner sep=0.4pt]{};
			\node at (0.25,0.5)[circle,fill,inner sep=0.5pt]{};
			\node at (0.3125,0.5)[circle,fill,inner sep=0.75pt]{};
			\node at (0.375,0.5)[circle,fill,inner sep=1pt]{};
			\node at (0.4375,0.5)[circle,fill,inner sep=1.25pt]{};
			\node at (0.5,0.5)[circle,fill,inner sep=1.5pt]{};
			\node at (0.5625,0.5)[circle,fill,inner sep=0.2pt]{};
			\node at (0.625,0.5)[circle,fill,inner sep=0.3pt]{};
			\node at (0.6875,0.5)[circle,fill,inner sep=0.4pt]{};
			\node at (0.75,0.5)[circle,fill,inner sep=0.5pt]{};
			\node at (0.8125,0.5)[circle,fill,inner sep=0.75pt]{};
			\node at (0.875,0.5)[circle,fill,inner sep=1pt]{};
			\node at (0.9375,0.5)[circle,fill,inner sep=1.25pt]{};
			\node at (1,0.5)[circle,fill,inner sep=1.5pt]{};
			
			\node at (1.0625,0.5)[circle,fill,inner sep=0.2pt]{};
			\node at (1.125,0.5)[circle,fill,inner sep=0.3pt]{};
			\node at (1.1875,0.5)[circle,fill,inner sep=0.4pt]{};
			\node at (1.25,0.5)[circle,fill,inner sep=0.5pt]{};
			\node at (1.3125,0.5)[circle,fill,inner sep=0.75pt]{};
			\node at (1.375,0.5)[circle,fill,inner sep=1pt]{};
			\node at (1.4375,0.5)[circle,fill,inner sep=1.25pt]{};
			\node at (1.5,0.5)[circle,fill,inner sep=1.5pt]{};
			\node at (1.5625,0.5)[circle,fill,inner sep=0.2pt]{};
			\node at (1.625,0.5)[circle,fill,inner sep=0.3pt]{};
			\node at (1.6875,0.5)[circle,fill,inner sep=0.4pt]{};
			\node at (1.75,0.5)[circle,fill,inner sep=0.5pt]{};
			\node at (1.8125,0.5)[circle,fill,inner sep=0.75pt]{};
			\node at (1.875,0.5)[circle,fill,inner sep=1pt]{};
			\node at (1.9375,0.5)[circle,fill,inner sep=1.25pt]{};
			\node at (2,0.5)[circle,fill,inner sep=1.5pt]{};
			
			\node at (2.0625,0.5)[circle,fill,inner sep=0.2pt]{};
			\node at (2.125,0.5)[circle,fill,inner sep=0.3pt]{};
			\node at (2.1875,0.5)[circle,fill,inner sep=0.4pt]{};
			\node at (2.25,0.5)[circle,fill,inner sep=0.5pt]{};
			\node at (2.3125,0.5)[circle,fill,inner sep=0.75pt]{};
			\node at (2.375,0.5)[circle,fill,inner sep=1pt]{};
			\node at (2.4375,0.5)[circle,fill,inner sep=1.25pt]{};
			\node at (2.5,0.5)[circle,fill,inner sep=1.5pt]{};
			\node at (2.5625,0.5)[circle,fill,inner sep=0.2pt]{};
			\node at (2.625,0.5)[circle,fill,inner sep=0.3pt]{};
			\node at (2.6875,0.5)[circle,fill,inner sep=0.4pt]{};
			\node at (2.75,0.5)[circle,fill,inner sep=0.5pt]{};
			\node at (2.8125,0.5)[circle,fill,inner sep=0.75pt]{};
			\node at (2.875,0.5)[circle,fill,inner sep=1pt]{};
			\node at (2.9375,0.5)[circle,fill,inner sep=1.25pt]{};
			\node at (3,0.5)[circle,fill,inner sep=1.5pt]{};
			
			\node at (3.0625,0.5)[circle,fill,inner sep=0.2pt]{};
			\node at (3.125,0.5)[circle,fill,inner sep=0.3pt]{};
			\node at (3.1875,0.5)[circle,fill,inner sep=0.4pt]{};
			\node at (3.25,0.5)[circle,fill,inner sep=0.5pt]{};
			\node at (3.3125,0.5)[circle,fill,inner sep=0.75pt]{};
			\node at (3.375,0.5)[circle,fill,inner sep=1pt]{};
			\node at (3.4375,0.5)[circle,fill,inner sep=1.25pt]{};
			\node at (3.5,0.5)[circle,fill,inner sep=1.5pt]{};
			\node at (3.5625,0.5)[circle,fill,inner sep=0.2pt]{};
			\node at (3.625,0.5)[circle,fill,inner sep=0.3pt]{};
			\node at (3.6875,0.5)[circle,fill,inner sep=0.4pt]{};
			\node at (3.75,0.5)[circle,fill,inner sep=0.5pt]{};
			\node at (3.8125,0.5)[circle,fill,inner sep=0.75pt]{};
			\node at (3.875,0.5)[circle,fill,inner sep=1pt]{};
			\node at (3.9375,0.5)[circle,fill,inner sep=1.25pt]{};
			\node at (4,0.5)[circle,fill,inner sep=1.5pt]{};
			
			\node (A) at (0,0) {};
			\node (B) at (1,0) {};
			\node (C) at (2,0) {};
			\node (D) at (3,0) {};
			\node (E) at (4,0) {};
			
			\draw [thick,
			decoration={brace,mirror,raise=0.15cm},
			decorate] (A) -- (B)
			node [pos=0.5,anchor=north,yshift=-0.2cm] {$I_{11}$};
			\draw [thick,
			decoration={brace,mirror,raise=0.15cm},
			decorate] (B) -- (C)
			node [pos=0.5,anchor=north,yshift=-0.2cm] {$I_{21}$};
			\draw [thick,
			decoration={brace,mirror,raise=0.15cm},
			decorate] (C) -- (D)
			node [pos=0.5,anchor=north,yshift=-0.2cm] {$I_{31}$};
			\draw [thick,
			decoration={brace,mirror,raise=0.15cm},
			decorate] (D) -- (E)
			node [pos=0.5,anchor=north,yshift=-0.2cm] {$I_{41}$};
			
			\node (F) at (1,2) {};
			\node (G) at (2,2) {};
			\node (H) at (3,2) {};
			\node (I) at (4,2) {};
			
			\draw [thick,
			decoration={brace,raise=0.6cm},
			decorate] (F) -- (G)
			node [pos=0.5,anchor=north,yshift=1.35cm] {$K^*_{11} = K^*_{21}$};
			
			\draw [thick,
			decoration={brace,raise=0.6cm},
			decorate] (H) -- (I)
			node [pos=0.5,anchor=north,yshift=1.35cm] {$K^*_{31} = K^*_{41}$};
			
			\end{tikzpicture}
			\vspace{-1em}
			\caption{The space $\XXX$ from Example~\ref{E4} (dot thickness $\leftrightsquigarrow$ size $2^{-r(i,j)}$).}
			\label{Pic2}
			\vspace{-1.5em}
		\end{figure}
	\end{center}
	
	\section{Maximal operators} \label{S3}
	
	For a collection $\SSS$ of sets $S \subseteq X$ such that $\mu(S) \in (0,\infty)$, one can define the corresponding \textit{maximal operator $\MMM_\SSS$} in a standard way. Namely, for every $p \in [1,\infty]$ and $f \in L^p(\XXX)$, we let
	\[
	\MMM_\SSS f(x) \coloneqq \sup_{x \in S \in \SSS} {\rm Avg}_{|f|}(S).
	\]
	Here we do not assume that $\SSS$ forms a differentiation basis; in particular, the supremum above may be taken over the empty set, in which case we put $\MMM_\SSS f(x) \coloneqq 0$. Note that, in general, $\MMM_\SSS f$ may not be measurable. Indeed, let us modify Example~\ref{E3} so that $\mu_K((i2^{-j}, 2^{-j}) \coloneqq 2^{-2j}$. Then $f \coloneqq \ind{K} \in L^p(\XXX)$ for every $p \in [1,\infty]$, while $\MMM_\BBB f = \ind{K \cup F}$ is not measurable. In order to avoid such issues, in this section we shall consider only maximal operators corresponding to countable collections $\SSS$.   
	
	For $p \in [1,\infty)$, we say that \textit{$\MMM_\SSS$ satisfies the weak-type $(p,p)$ inequality} if
	\begin{align} \label{weak-type}
		\lambda^p \cdot \mu(\{ x \in X : \MMM_\SSS f(x) > \lambda\} ) \leq C_p^p \|f\|^p_{L^p(\XXX)}   
	\end{align}
	holds for every $f \in L^p(\XXX)$ and $\lambda \in (0, \infty)$, where $C_p \in [0,\infty)$ is independent of $f$ and $\lambda$. We write $\|\MMM_\SSS\|_{L^p \to L^{p,\infty}}$ for the smallest constant $C_p \in [0, \infty]$ such that \eqref{weak-type} holds.
	
	The following result is well-known, but we show its proof for the sake of completeness.
	
	\begin{lemma} \label{L0}
		Let $\XXX$ be a metric measure space and fix $p \in [1,\infty)$. Assume that $L^p(\XXX)$ contains a dense subset consisting of continuous functions. If $\BBB$ is a countable uncentered differentiation basis in $\XXX$ and $\MMM_\BBB$ satisfies the weak-type $(p,p)$ inequality, then $p \in {\rm diff}(\BBB)$.  
	\end{lemma}
	
	\begin{proof}
		Fix $f \in L^p(\XXX)$ and $\varepsilon \in (0, \infty)$. It suffices to show that
		\[
		E_\varepsilon \coloneqq \big\{ x \in X : |\overline{f}(x) - f(x)| > 2 \varepsilon \text{ or } |\underline{f}(x) - f(x)| > 2 \varepsilon \big\}
		\]
		has measure zero. Note that $\overline{f}$ and $\underline{f}$ are measurable by the assumptions imposed on $\BBB$.
		
		For every $j \in \NN$, let $f_j \in L^p(\XXX)$ be a continuous function such that $\|f - f_j\|_{L^p(\XXX)} \leq 1/j$. Assume that $(S_n)_{n \in \NN}$ contracts to some $x \in X$. We write
		\[
		{\rm Avg}_f(S_n) = {\rm Avg}_{f_j}(S_n) + {\rm Avg}_{f-f_j}(S_n)   
		\]
		and observe that
		$
		\lim_{n \to \infty} {\rm Avg}_{f_j}(S_n) = f_j(x)
		$
		because $f_j$ is continuous. On the other hand,
		\[
		\sup_{n \in \NN} |{\rm Avg}_{f-f_j}(S_n)| \leq \MMM_\BBB (f-f_j)(x).
		\]
		Thus, we conclude that $E_\varepsilon$ is contained in the set
		\[
		E_\varepsilon^* \coloneqq \big\{ x \in X : |(f-f_j)(x)| > \varepsilon \text{ or } \MMM_\BBB (f-f_j)(x) > \varepsilon \big\}.    
		\]
		By Chebyshev's inequality and the weak-type $(p,p)$ inequality \eqref{weak-type}, we have
		\[
		\mu(E_\varepsilon) \leq \mu(E_\varepsilon^*) \leq \varepsilon^{-p} (1+C_p^p) \| f - f_j \|_{L^p(\XXX)}^p \leq (\varepsilon j)^{-p} (1+C_p^p).
		\]
		Therefore, by letting $j \to \infty$, we obtain $\mu(E_\varepsilon) = 0$ and the proof is complete.  
	\end{proof}
	
	By Lemma~\ref{L0}, a possible way to show that $p \in {\rm diff}(\BBB)$ is to study the maximal operator $\MMM_\BBB$. The following result allows us to verify \eqref{weak-type} for certain bases introduced in Section~\ref{S4}. 
	
	\begin{proposition}\label{P2}
		Let $\XXX$ be a metric measure space and fix $p \in [1,\infty)$. For every $n \in \NN$, let $\SSS_n$ be a countable collection such that $\MMM_{\SSS_n}$ satisfies \eqref{weak-type} with a constant $C_p \in (0, \infty)$ independent of $n$. Assume that if $S_1 \in \SSS_{n_1}$ and $S_2 \in \SSS_{n_2}$ for distinct $n_1, n_2 \in \NN$, then 
		\[
		S_1 \subseteq S_2
		\quad \text{or} \quad
		S_2 \subseteq S_1
		\quad \text{or} \quad
		S_1 \cap S_2 = \emptyset. 
		\]
		Let $\SSS \coloneqq \bigcup_{n \in \NN} \SSS_n$. Then $\MMM_{\SSS}$ satisfies \eqref{weak-type} with the same constant $C_p$.   
	\end{proposition}
	
	\begin{proof}
		Fix $f \in L^p(\XXX)$ and $\lambda \in (0,\infty)$.  For every $n \in \NN$, set 
		\begin{align*}
			E_\lambda & \coloneqq \{x \in X : \MMM_\SSS f > \lambda \}, \\
			E_{\lambda,n} & \coloneqq \{x \in X : \MMM_{\SSS_n} f > \lambda \}.
		\end{align*} 
		Clearly, $E_\lambda = \bigcup_{n \in \NN} E_{\lambda,n}$ so that $E_\lambda$ is measurable. Our goal is to prove the inequality 
		\[
		\lambda^p \cdot \mu(E_\lambda) \leq C_p^p \|f\|^p_{L^p(\XXX)}.
		\]
		
		Fix $\delta > 0$. For every $n \in \NN$, since $\SSS_n$ is countable and
		\[
		\lambda^p \cdot \mu(E_{\lambda,n}) \leq C_p^p \|f\|^p_{L^p(\XXX)} < \infty,
		\] 
		one can find a positive integer $I_n \in \NN$ and a finite family $\{S_{n,1}, \dots, S_{n,I_n}\} \subseteq \SSS_n$ such that 
		\[
		\mu(E_{\lambda,n}) \leq \mu\big(S_{n,1} \cup \cdots \cup S_{n,I_n} \big) + \delta / 2^n
		\]
		and ${\rm Avg}_{|f|}(S_{n,i}) > \lambda$ for every $i \in \{1, \dots, I_n\}$. Thus, it suffices to prove the inequality
		\begin{align} \label{max1}
			\lambda^p \cdot \mu\Big(\bigcup_{n=1}^{N} \bigcup_{i=1}^{I_n} S_{n,i} \Big) \leq C_p^p \|f\|^p_{L^p(\XXX)}   
		\end{align}
		for every $N \in \NN$.
		
		Fix $N \in \NN$. For every $n \in \{1,\dots,N\}$, we can choose an integer $I_n^* \in \{0,1,\dots, I_n\}$ and a subfamily 
		$\{S_{n,1}^*, \dots, S^*_{n,I^*_n}\} \subseteq \{S_{n,1}, \dots, S_{n,I_n}\}$ in such a way that the family
		\[
		\SSS^* \coloneqq \big\{ S^*_{n,i} : n \in \{1,\dots, N\}, \, i \in \{1,\dots, I_n^*\} \big\}
		\]
		consists of all maximal elements (with possible repetitions being removed) of the family
		\[
		\big\{ S_{n,i} : n \in \{1,\dots, N\}, \, i \in \{1,\dots, I_n\} \big\}
		\]
		ordered by the inclusion relation. For every $n \in \{1, \dots, N\}$, let $E_{\lambda,n}^* \coloneqq \bigcup_{i=1}^{I_n^*} S^*_{n,i}$. Then the sets $E_{\lambda,n}^*$ are disjoint and so 
		\begin{align} \label{max2}
			\mu \Big(\bigcup_{n=1}^{N} \bigcup_{i=1}^{I_n} S_{n,i} \Big) = \sum_{n=1}^N \mu(E_{\lambda,n}^*).
		\end{align}
		Indeed, for distinct $n_1, n_2 \in \{1,\dots, N\}$, if 
		$
		i_1 \in \{1, \dots, I_{n_1}^*\}
		$ and $
		i_2 \in \{1, \dots, I_{n_2}^*\},
		$
		then by maximality we cannot have neither $S^*_{n_1,i_1} \subseteq S^*_{n_2,i_2}$ nor $S^*_{n_2,i_2} \subseteq S^*_{n_1,i_1}$. Hence, we must have
		$S^*_{n_1,i_1} \cap S^*_{n_2,i_2} = \emptyset$ by our initial assumption on the collections $\SSS_{n_1}$ and $\SSS_{n_2}$.
		
		For every $n \in \{1,\dots,N\}$, set $f_n \coloneqq f \cdot \ind{E_{\lambda,n}^*}$. Since $E_{\lambda,n}^*$ are disjoint, we obtain
		\begin{align} \label{max3}
			\sum_{n=1}^N \|f_n\|^p_{L^p(\XXX)} \leq \| f \|^p_{L^p(\XXX)}.  
		\end{align}
		Moreover, for every $i \in \{1, \dots, I_n^*\}$, we have 
		$
		{\rm Avg}_{|f_n|}(S^*_{n,i}) = {\rm Avg}_{|f|}(S^*_{n,i}).
		$
		Hence,
		\[
		E_{\lambda,n}^* \subseteq \{x \in X : \MMM_{\SSS_n} f_n(x) > \lambda\}.
		\]
		By our initial assumption on $\MMM_{\SSS_n}$, this gives
		\begin{align} \label{max4}
			\lambda^p \cdot \mu(E_{\lambda,n}^*) \leq C_p^p \|f_n\|^p_{L^p(\XXX)}.
		\end{align}
		Combining \eqref{max2}--\eqref{max4}, we obtain \eqref{max1} and the proof is complete. 
	\end{proof}
	
	\section{Infinite-dimensional torus} \label{S4}
	
	In order to illustrate Theorem~\ref{T1} with suitable examples, we introduce a very specific class of differentiation bases in the infinite-dimensional torus $\tom$.
	
	By $\TT$ we mean the quotient space $\RR / \ZZ$ which is naturally interpreted as the interval $[0,1]$ with its endpoints identified.
	We define the infinite-dimensional torus $\TT^\omega  \coloneqq \TT \times \TT \times \cdots$ as the product of countably many copies of $\TT$ indexed by $\NN$. Thus, $\tom$ is a group with ${\bf 0} \coloneqq (0,0,\dots)$ being its identity. For every $d\in\mathbb{N}$, we write $\tom=\mathbb{T}^d\times\mathbb{T}^{d,\omega}$ so that $\mathbb{T}^{d,\omega}$ is a copy of $\tom$ referring to the coordinates with indices larger than $d$. 
	
	Note that $\tom$ with the usual product topology is compact and metrizable. Indeed, the former follows from Tychonoff's theorem, while the latter holds because the function
	\[
	\rho_{\tom} ((x_1, x_2, \dots),(y_1, y_2, \dots)) \coloneqq \sum_{d=1}^\infty \frac{\min\{|x_d-y_d|,1-|x_d-y_d|\}}{2^d}
	\]
	defined on $\tom \times \tom$ determines a metric which is compatible with the product topology.  
	
	The $\sigma$-algebra of Borel subsets of $\tom$ is generated by rectangles of the form $\prod_{d=1}^\infty I_d$, where each $I_d$ is an interval in $\TT$ and, for some $d_0 \in \NN$, we have $I_d=\TT$ if $d> d_0$. The normalized Haar measure ${\rm d} x$ on $\tom$ is the product of countably many copies of the normalized Lebesgue measures on $\TT$. For a Borel set $E\subseteq \tom$, we denote its measure by
	\[
	|E| \coloneqq \int_{\tom} \ind{E}(x) \, {\rm d}x.
	\] 
	As mentioned before, by using a standard procedure we can make the space complete. 
	
	The infinite-dimensional torus $\tom$ still retains some of the differentiation properties of the finite-dimensional tori. For example, every density basis in $\tom$ differentiates $L^\infty(\tom)$, (cf.~\cite[Chapter~III, Theorem~1.4]{dG75}), because Lusin's Theorem still holds true in $\tom$. Nevertheless, in the context of differentiation bases consisting of rectangles with sides parallel to the coordinate axes, the infinite-dimensional torus $\tom$ exhibits a broader variety of behaviors that are absent in the finite-dimensional case.
	
	Historically, the two most important differentiation bases in $\tom$ are the Rubio de Francia bases, classical $\mathcal{R}$ and dyadic $\mathcal{R}_0$. While $\mathcal{R}$ is the object originally devised by Jos\'e Luis Rubio de Francia, $\mathcal{R}_0$ appears in his work somewhat indirectly, as a special case of dyadic bases in the abstract setting of locally compact groups (see \cite{RdF78}). Below we briefly describe how $\mathcal{R}$ and $\mathcal{R}_0$ look like (for formal definitions, see \cite{FR20} or the doctoral dissertation \cite{Fe19}). 
	
	For every $m \in \NN$, let $R_m \coloneqq \{0, \frac{1}{m}, \dots, \frac{m-1}{m}\}$. We also need two auxiliary sequences:
	\begin{itemize}
		\item the sequence $(H_m)_{m \in \NN}$ consisting of finite subgroups $H_1 \subseteq H_2 \subseteq \cdots \subseteq \tom$ such that $\bigcup_{m\in \mathbb{N}}H_m$ is dense in $\tom$ and $[H_{m+1}:H_m]=2$ for every $m \in \NN$,
		\item the sequence $(V_m)_{m \in \NN}$ consisting of subsets $\tom \supseteq V_1 \supseteq V_2 \supseteq \cdots$ such that each $V_m$ is the interior of the fundamental domain of the quotient group $\tom / H_m$.  
	\end{itemize} 
	In Table~\ref{Tab} below, we list the first few objects. By ${\bf 0}^{d, \omega}$ we mean the zero vector in $\TT^{d, \omega}$.
	
	\begin{table}[H]	
		\begin{center}
			\caption{Objects used to define the Rubio de Francia bases.}
			\scalebox{1.0}{
				\begin{tabular}{ r l l }
					\hline
					$m$ & $H_m$ & $V_m$ \\
					\hline
					$1$ & $R_2 \times \{{\bf 0}^{1,\omega}\}$
					& $(0,1/2) \times \TT^{1, \omega}$ \\ 
					$2$ & $R_2 \times R_2 \times \{{\bf 0}^{2,\omega}\}$ 
					& $(0,1/2) \times (0,1/2) \times \TT^{2, \omega}$ \\ 
					$3$ & $R_4 \times R_2 \times \{{\bf 0}^{2,\omega}\}$ & $(0,1/4) \times (0,1/2) \times \TT^{2, \omega}$ \\
					$4$ & $R_4 \times R_4 \times \{{\bf 0}^{2,\omega}\}$ & $(0,1/4) \times (0,1/4) \times \TT^{2, \omega}$ \\ 
					$5$ & $R_4 \times R_4 \times R_2 \times \{{\bf 0}^{3,\omega}\}$ & $(0,1/4) \times (0,1/4) \times (0,1/2) \times\TT^{3, \omega}$ \\ 
					$6$ & $R_4 \times R_4 \times R_4 \times \{{\bf 0}^{3,\omega}\}$ & $(0,1/4) \times (0,1/4) \times (0,1/4) \times\TT^{3, \omega}$ \\
					$7$ & $R_8 \times R_4 \times R_4 \times \{{\bf 0}^{3,\omega}\}$ & $(0,1/8) \times (0,1/4) \times (0,1/4) \times\TT^{3, \omega}$ \\ 
					$8$ & $R_8 \times R_8 \times R_4 \times \{{\bf 0}^{3,\omega}\}$ & $(0,1/8) \times (0,1/8) \times (0,1/4) \times\TT^{3, \omega}$ \\ 
					$9$ & $R_8 \times R_8 \times R_8 \times \{{\bf 0}^{3,\omega}\}$ & $(0,1/8) \times (0,1/8) \times (0,1/8) \times\TT^{3, \omega}$ \\
					$10$ & $R_8 \times R_8 \times R_8 \times R_2 \times \{{\bf 0}^{4,\omega}\}$ & $(0,1/8) \times (0,1/8) \times (0,1/8) \times (0,1/2) \times\TT^{4, \omega}$ \\
					$\cdots$ & $\cdots$ & $\cdots$ \\
					\hline
			\end{tabular}}
			\label{Tab}
		\end{center}
	\end{table} 
	
	According to this notation, the Rubio de Francia bases are defined by
	\begin{align*}
		\RRR_0  & \coloneqq \{ t + V_m : m \in \NN, \, t \in H_m \}, \\
		\RRR  & \coloneqq \{ t + V_m : m \in \NN, \, t \in \TT^\omega \}.
	\end{align*}
	In Figure~\ref{Pic3} below, we also show the first eight sets $V_m$ (or, to be absolutely precise, the shapes of their projections onto the first three coordinates of $\tom$).
	
	\begin{center}
		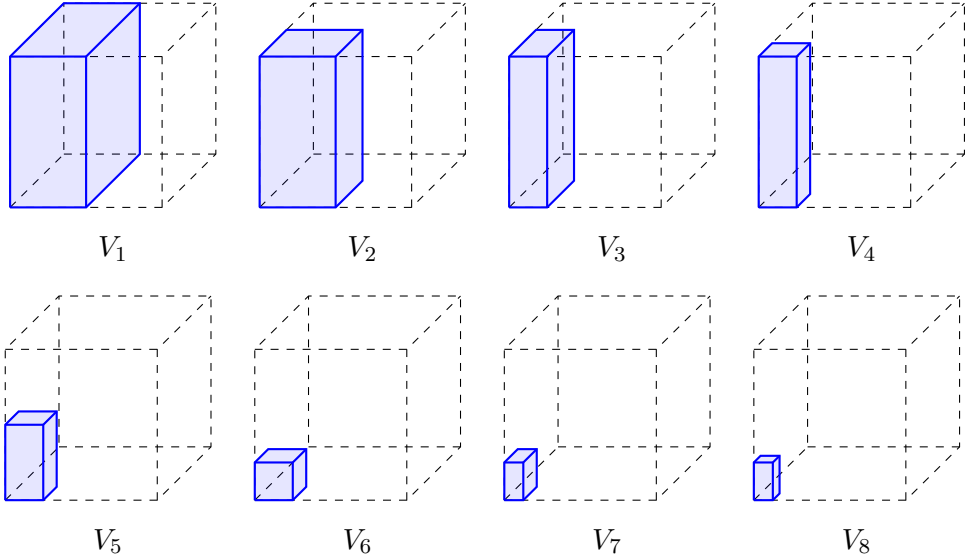
\begin{figure}[H]
			\begin{tikzpicture}
			[scale=1]
			
			\draw[dashed] (0,0) -- (2,0) -- (2+0.5*1.414,0.5*1.414) -- (2+0.5*1.414,2+0.5*1.414) -- (0.5*1.414,2+0.5*1.414) -- (0,2) -- (0,0);
			\draw[dashed] (2,0) -- (2,2) -- (2+0.5*1.414,2+0.5*1.414); 
			\draw[dashed] (2,2) -- (0,2);
			\draw[dashed] (0,0) -- (0.5*1.414,0.5*1.414) -- (0.5*1.414,2+0.5*1.414); 
			\draw[dashed] (0.5*1.414,0.5*1.414) -- (2+0.5*1.414,0.5*1.414);
			
			\draw[blue,thick,fill=blue,fill opacity=0.1] (0,0) -- (1,0) -- (1+0.5*1.414,0.5*1.414) -- (1+0.5*1.414,2+0.5*1.414) -- (0.5*1.414,2+0.5*1.414) -- (0,2) -- (0,0);
			\draw[blue,thick] (1,0) -- (1,2) -- (1+0.5*1.414,2+0.5*1.414); 
			\draw[blue,thick] (1,2) -- (0,2);
			
			\node at (1+0.25*1.414,-0.25)[below]{$V_1$};
			\end{tikzpicture}
			\hspace{3mm}
			\begin{tikzpicture}
			[scale=1]
			
			\draw[dashed] (0,0) -- (2,0) -- (2+0.5*1.414,0.5*1.414) -- (2+0.5*1.414,2+0.5*1.414) -- (0.5*1.414,2+0.5*1.414) -- (0,2) -- (0,0);
			\draw[dashed] (2,0) -- (2,2) -- (2+0.5*1.414,2+0.5*1.414); 
			\draw[dashed] (2,2) -- (0,2);
			\draw[dashed] (0,0) -- (0.5*1.414,0.5*1.414) -- (0.5*1.414,2+0.5*1.414); 
			\draw[dashed] (0.5*1.414,0.5*1.414) -- (2+0.5*1.414,0.5*1.414);
			
			\draw[blue,thick,fill=blue,fill opacity=0.1] (0,0) -- (1,0) -- (1+0.25*1.414,0.25*1.414) -- (1+0.25*1.414,2+0.25*1.414) -- (0.25*1.414,2+0.25*1.414) -- (0,2) -- (0,0);
			\draw[blue,thick] (1,0) -- (1,2) -- (1+0.25*1.414,2+0.25*1.414); 
			\draw[blue,thick] (1,2) -- (0,2);
			
			\node at (1+0.25*1.414,-0.25)[below]{$V_2$};
			\end{tikzpicture}
			\hspace{3mm}
			\begin{tikzpicture}
			[scale=1]
			
			\draw[dashed] (0,0) -- (2,0) -- (2+0.5*1.414,0.5*1.414) -- (2+0.5*1.414,2+0.5*1.414) -- (0.5*1.414,2+0.5*1.414) -- (0,2) -- (0,0);
			\draw[dashed] (2,0) -- (2,2) -- (2+0.5*1.414,2+0.5*1.414); 
			\draw[dashed] (2,2) -- (0,2);
			\draw[dashed] (0,0) -- (0.5*1.414,0.5*1.414) -- (0.5*1.414,2+0.5*1.414); 
			\draw[dashed] (0.5*1.414,0.5*1.414) -- (2+0.5*1.414,0.5*1.414);
			
			\draw[blue,thick,fill=blue,fill opacity=0.1] (0,0) -- (0.5,0) -- (0.5+0.25*1.414,0.25*1.414) -- (0.5+0.25*1.414,2+0.25*1.414) -- (0.25*1.414,2+0.25*1.414) -- (0,2) -- (0,0);
			\draw[blue,thick] (0.5,0) -- (0.5,2) -- (0.5+0.25*1.414,2+0.25*1.414); 
			\draw[blue,thick] (0.5,2) -- (0,2);
			
			\node at (1+0.25*1.414,-0.25)[below]{$V_3$};
			\end{tikzpicture}
			\hspace{3mm}
			\begin{tikzpicture}
			[scale=1]
			
			\draw[dashed] (0,0) -- (2,0) -- (2+0.5*1.414,0.5*1.414) -- (2+0.5*1.414,2+0.5*1.414) -- (0.5*1.414,2+0.5*1.414) -- (0,2) -- (0,0);
			\draw[dashed] (2,0) -- (2,2) -- (2+0.5*1.414,2+0.5*1.414); 
			\draw[dashed] (2,2) -- (0,2);
			\draw[dashed] (0,0) -- (0.5*1.414,0.5*1.414) -- (0.5*1.414,2+0.5*1.414); 
			\draw[dashed] (0.5*1.414,0.5*1.414) -- (2+0.5*1.414,0.5*1.414);
			
			\draw[blue,thick,fill=blue,fill opacity=0.1] (0,0) -- (0.5,0) -- (0.5+0.125*1.414,0.125*1.414) -- (0.5+0.125*1.414,2+0.125*1.414) -- (0.125*1.414,2+0.125*1.414) -- (0,2) -- (0,0);
			\draw[blue,thick] (0.5,0) -- (0.5,2) -- (0.5+0.125*1.414,2+0.125*1.414); 
			\draw[blue,thick] (0.5,2) -- (0,2);
			
			\node at (1+0.25*1.414,-0.25)[below]{$V_4$};
			\end{tikzpicture}
			
			\vspace{3mm}
			
			\begin{tikzpicture}
			[scale=1]
			
			\draw[dashed] (0,0) -- (2,0) -- (2+0.5*1.414,0.5*1.414) -- (2+0.5*1.414,2+0.5*1.414) -- (0.5*1.414,2+0.5*1.414) -- (0,2) -- (0,0);
			\draw[dashed] (2,0) -- (2,2) -- (2+0.5*1.414,2+0.5*1.414); 
			\draw[dashed] (2,2) -- (0,2);
			\draw[dashed] (0,0) -- (0.5*1.414,0.5*1.414) -- (0.5*1.414,2+0.5*1.414); 
			\draw[dashed] (0.5*1.414,0.5*1.414) -- (2+0.5*1.414,0.5*1.414);
			
			\draw[blue,thick,fill=blue,fill opacity=0.1] (0,0) -- (0.5,0) -- (0.5+0.125*1.414,0.125*1.414) -- (0.5+0.125*1.414,1+0.125*1.414) -- (0.125*1.414,1+0.125*1.414) -- (0,1) -- (0,0);
			\draw[blue,thick] (0.5,0) -- (0.5,1) -- (0.5+0.125*1.414,1+0.125*1.414); 
			\draw[blue,thick] (0.5,1) -- (0,1);
			
			\node at (1+0.25*1.414,-0.25)[below]{$V_5$};
			\end{tikzpicture}
			\hspace{3mm}
			\begin{tikzpicture}
			[scale=1]
			
			\draw[dashed] (0,0) -- (2,0) -- (2+0.5*1.414,0.5*1.414) -- (2+0.5*1.414,2+0.5*1.414) -- (0.5*1.414,2+0.5*1.414) -- (0,2) -- (0,0);
			\draw[dashed] (2,0) -- (2,2) -- (2+0.5*1.414,2+0.5*1.414); 
			\draw[dashed] (2,2) -- (0,2);
			\draw[dashed] (0,0) -- (0.5*1.414,0.5*1.414) -- (0.5*1.414,2+0.5*1.414); 
			\draw[dashed] (0.5*1.414,0.5*1.414) -- (2+0.5*1.414,0.5*1.414);
			
			\draw[blue,thick,fill=blue,fill opacity=0.1] (0,0) -- (0.5,0) -- (0.5+0.125*1.414,0.125*1.414) -- (0.5+0.125*1.414,0.5+0.125*1.414) -- (0.125*1.414,0.5+0.125*1.414) -- (0,0.5) -- (0,0);
			\draw[blue,thick] (0.5,0) -- (0.5,0.5) -- (0.5+0.125*1.414,0.5+0.125*1.414); 
			\draw[blue,thick] (0.5,0.5) -- (0,0.5);
			
			\node at (1+0.25*1.414,-0.25)[below]{$V_6$};
			\end{tikzpicture}
			\hspace{3mm}
			\begin{tikzpicture}
			[scale=1]
			
			\draw[dashed] (0,0) -- (2,0) -- (2+0.5*1.414,0.5*1.414) -- (2+0.5*1.414,2+0.5*1.414) -- (0.5*1.414,2+0.5*1.414) -- (0,2) -- (0,0);
			\draw[dashed] (2,0) -- (2,2) -- (2+0.5*1.414,2+0.5*1.414); 
			\draw[dashed] (2,2) -- (0,2);
			\draw[dashed] (0,0) -- (0.5*1.414,0.5*1.414) -- (0.5*1.414,2+0.5*1.414); 
			\draw[dashed] (0.5*1.414,0.5*1.414) -- (2+0.5*1.414,0.5*1.414);
			
			\draw[blue,thick,fill=blue,fill opacity=0.1] (0,0) -- (0.25,0) -- (0.25+0.125*1.414,0.125*1.414) -- (0.25+0.125*1.414,0.5+0.125*1.414) -- (0.125*1.414,0.5+0.125*1.414) -- (0,0.5) -- (0,0);
			\draw[blue,thick] (0.25,0) -- (0.25,0.5) -- (0.25+0.125*1.414,0.5+0.125*1.414); 
			\draw[blue,thick] (0.25,0.5) -- (0,0.5);
			
			\node at (1+0.25*1.414,-0.25)[below]{$V_7$};
			\end{tikzpicture}
			\hspace{3mm}
			\begin{tikzpicture}
			[scale=1]
			
			\draw[dashed] (0,0) -- (2,0) -- (2+0.5*1.414,0.5*1.414) -- (2+0.5*1.414,2+0.5*1.414) -- (0.5*1.414,2+0.5*1.414) -- (0,2) -- (0,0);
			\draw[dashed] (2,0) -- (2,2) -- (2+0.5*1.414,2+0.5*1.414); 
			\draw[dashed] (2,2) -- (0,2);
			\draw[dashed] (0,0) -- (0.5*1.414,0.5*1.414) -- (0.5*1.414,2+0.5*1.414); 
			\draw[dashed] (0.5*1.414,0.5*1.414) -- (2+0.5*1.414,0.5*1.414);
			
			\draw[blue,thick,fill=blue,fill opacity=0.1] (0,0) -- (0.25,0) -- (0.25+0.0625*1.414,0.0625*1.414) -- (0.25+0.0625*1.414,0.5+0.0625*1.414) -- (0.0625*1.414,0.5+0.0625*1.414) -- (0,0.5) -- (0,0);
			\draw[blue,thick] (0.25,0) -- (0.25,0.5) -- (0.25+0.0625*1.414,0.5+0.0625*1.414); 
			\draw[blue,thick] (0.25,0.5) -- (0,0.5);
			
			\node at (1+0.25*1.414,-0.25)[below]{$V_8$};
			\end{tikzpicture}
			\vspace{-1em}
			\caption{The first eight elements $V_m$ in $\TT^{3} \times \TT^{3, \omega}$.}
			\label{Pic3}
			\vspace{-1.5em}
		\end{figure}
	\end{center}
	
	Both $\RRR_0$ and $\RRR$ are considered as uncentered bases. Interestingly, ${\rm diff}(\RRR_0) = [1,\infty]$ (see \cite[Corollary~16]{FR20}), while ${\rm diff}(\RRR) = \emptyset$ (see \cite[Theorem~1.1]{Ko21}). Although $\RRR$ is uncountable, we can conveniently ignore this fact, referring only to its differentiation properties \cite[Theorem~1.1]{Ko21} and not to any results concerning the maximal operator $\MMM_\RRR$. However, we recall that $\RRR$ can be replaced by a countable collection $\RRR' \subseteq \RRR$ such that $\MMM_\RRR f = \MMM_{\RRR'} f$ for every $f \in L^1(\tom)$ (see \cite[Lemma~3.2]{KMPRR23}).
	
	In \cite{KMPRR23, KRR23}, the authors studied various properties of the maximal operators $\MMM_{\RRR'}$ corresponding to some specific intermediate uncentered bases $\RRR_0 \subseteq \RRR' \subseteq \RRR$. In particular, their considerations led the to define the following object. 
	
	\begin{definition}\label{epsn}
		Fix $\varepsilon \in (0, \frac{1}{2}]$ and $d \in \mathbb{N}$. 
		Let $Q_0 \in \mathcal{R}_0$ be such that $|Q_0| \leq 2^{-(d-1)^2-1}$ and so $Q_0$ has at least $d$ nonfree coordinates (cf.~\cite[Subsection~2.3]{KMPRR23}). For every $i \in \{1, \dots, d\}$, let $Q_{i}$ be $Q_0$ translated to the right with respect to the $i$-th coordinate in such a way that $| Q_{i} \cap Q_0| = \varepsilon |Q_0|$. Then the family $\SSS_0 \coloneqq \{Q_0, Q_{1}, \dots, Q_{d}\}$ is called an \emph{$(\varepsilon, d)$-configuration around $Q_0$} (see Figure~\ref{Pic4.1}).
	\end{definition}
	
	\begin{center}
		\begin{figure}[H]
			\begin{tikzpicture}
			[scale=0.6]
			
			\draw[thick] (0,0) rectangle (8,8);
			\config{0}{8}{0}{8}{0.25}{blue}
			
			\node at (0.75,0.75)[color=blue]{\large$Q_0$};
			\node at (6.25,0.75)[color=blue]{\large$Q_1$};
			\node at (0.75,6.25)[color=blue]{\large$Q_2$};
			
			\end{tikzpicture}
			\vspace{-1em}
			\caption{The $(\frac{1}{4},2)$-configuration around $(0,\frac{1}{2})^2 \times \TT^{2,\omega}$.}
			\label{Pic4.1}
			\vspace{-1.5em}
		\end{figure}
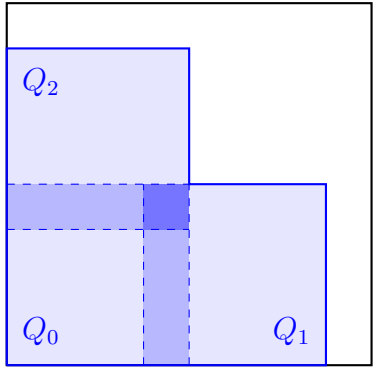
	\end{center}
	
	Let us recall \cite[Theorem~1.2]{KRR23}. Here, we write $A \simeq B$ if there exists a constant $C \in [1,\infty)$, independent of all parameters involved, such that $C^{-1}B \leq A \leq CB$. 
	
	\begin{theorem} \label{T2}
		Fix $\varepsilon \in (0, \frac{1}{2}]$ and $d \in \mathbb{N}$. Let $\mathcal{S}_0$ be an $(\varepsilon, d)$-configuration around $Q_0 \in \mathcal{R}_0$. For every $p \in (1, \infty)$, define $A_p(\mathcal{S}_0) \coloneqq \varepsilon d^{1/p}$ and $p_* \coloneqq \frac{p}{p-1}$. Then we have
		\[
		\|\mathcal{M}_{\mathcal{S}_0}\|_{L^p \to L^{p,\infty}} \simeq
		\begin{cases}
		1 & \text{if } A_p(\mathcal S_0) \in (0,p_* e^{-p_*}], \\
		p_* \log^{-1}(p_* \slash A_p(\mathcal S_0)) & \text{if } A_p(\mathcal S_0) \in [p_* e^{-p_*}, p_* e^{-1}], \\
		e A_p(\mathcal S_0) & \text{if } A_p(\mathcal S_0) \in [p_* e^{-1}, \infty).
		\end{cases} 
		\]
	\end{theorem}
	
	To be absolutely precise, in \cite[Theorem~1.2]{KRR23} the family $\SSS'_0 \coloneqq \SSS_0 \setminus \{Q_0\}$ was studied instead, but obviously we have $\|\mathcal{M}_{\mathcal{S}'_0}\|_{L^p \to L^{p,\infty}} \simeq \|\mathcal{M}_{\mathcal{S}_0}\|_{L^p \to L^{p,\infty}}$ uniformly in all parameters.
	In order to make the exposition more instructive, we also note that the right-hand side increases with $A_p(\mathcal S_0)$, while the two extreme cases, $\|\mathcal{M}_{\mathcal{S}_0}\|_{L^p \to L^{p,\infty}} \simeq 1$ and $\|\mathcal{M}_{\mathcal{S}_0}\|_{L^p \to L^{p,\infty}} \simeq e A_p(\mathcal S_0)$, are related to the following two simple observations:
	\begin{itemize}
		\item if $f = \ind{Q_0}$, then $\mathcal{M}_{\mathcal{S}_0} f \geq \ind{Q_0}$,
		\item if $f = \ind{Q_0}$, then $\mathcal{M}_{\mathcal{S}_0} f \geq \varepsilon \ind{Q_0 \cup Q_1 \cup \dots \cup Q_d}$, while $|Q_0 \cup Q_1 \cup \dots \cup Q_d| \simeq d |Q_0|$.
	\end{itemize}
	
	Our strategy is to combine Theorem~\ref{T2} with Lemma~\ref{L0} and Proposition~\ref{P2} in order to produce many differentiation bases with very peculiar properties. Before that, we explain how to cover a given rectangle $U \subseteq \tom$ by many $(\varepsilon, d)$-configurations of a particular form.   
	
	\begin{lemma} \label{L1}
		Fix $d \in \NN$, $\varepsilon \in (0,\frac{1}{2}]$, $m \in \NN$, and a rectangle $U \subseteq \tom$. Then $U$ can be covered, up to a set of measure zero, by a family $\{E_n : n \in \NN\}$ of disjoint sets $E_n \coloneqq \bigcup_{Q \in \SSS_n} Q$ for some $(\varepsilon, d)$-configurations $\SSS_n$ such that $|E_{n_1}| = |E_{n_2}|$ whenever $\lceil \frac{n_1}{2^m}\rceil = \lceil \frac{n_2}{2^m} \rceil$.      
	\end{lemma}
	\begin{proof}
		For any $k \in \NN$, denote 
		\begin{align*}
			P_k & \coloneqq H_{k^2+k} = \{0 / 2^k, 1/2^k, \dots, (2^k-1)/2^k \}^{k+1} \times \{{\bf 0}^{k+1, \omega}\}, \\
			Q_k & \coloneqq V_{k^2+k} = (0,2^{-k})^{k+1} \times \TT^{k+1, \omega}.
		\end{align*}
		We split $U$, up to a set of measure zero, into disjoint subsets $S \in \RRR_0$ of the form $g + Q_k$ for some $k \geq m+d$ and $g \in P_k$ depending on $S$. Next, we split each $S$ into $2^{2k+2}$ sets $S^* \in \RRR_0$ of the form $g^* + Q_{k+1}$ for some $g^* \in P_{k+1}$. Since $k \geq m$, if each of the $2^{2k+2}$ sets $S^*$ is covered in the same way, a suitable enumeration of the $(\varepsilon, d)$-configurations $\SSS_n$ will ensure that $|E_{n_1}| = |E_{n_2}|$ whenever $\lceil \frac{n_1}{2^m}\rceil = \lceil \frac{n_2}{2^m} \rceil$ (see Figure~\ref{Pic4.3}).
		
		It remains to cover each $S^*$ in a suitable way. Consider the $(\varepsilon, d)$-configuration $\SSS$ around 
		\[
		Q^* \coloneqq g^* + (0,2^{-k-2})^{d} \times (0,2^{-k-1})^{k+2-d} \times \TT^{k+2, \omega}.
		\]
		Note that $Q^* \in \RRR_0$ by $k \geq d$. Moreover, we have $|\bigcup_{Q \in \SSS} Q| = c |S^*|$ with 
		\[
		c \coloneqq 2^{-d} (1+d-\varepsilon d) \in (0,1).
		\]
		Next, as before, we split $S^* \setminus \bigcup_{Q \in \SSS} Q$ into disjoint subsets of the form $g' + Q_{k'}$, where $k' \geq m+d$ and $g' \in P_{k'}$. Here we take the maximal subsets to ensure that each of the $2^{2k+2}$ sets of the form $S^* \setminus \bigcup_{Q \in \SSS} Q$ is split in the same way. Then, for every subset of $S^* \setminus \bigcup_{Q \in \SSS} Q$ obtained in this way, we repeat the procedure described above. The portion of $S^*$ covered in this step has measure $c(1-c)|S^*|$. We continue in the same way and the portion of $S^*$ covered in the next step has measure $c(1-c)^2|S^*|$. After infinitely many such steps we cover $S^*$ up to a set of measure zero, since $c + c(1-c) + c(1-c)^2 + \cdots = 1$. 
	\end{proof}
	
	\begin{center}
		\begin{figure}[H]
			\begin{tikzpicture}
			[scale=0.8]
			
			\draw[thick] (0,0) rectangle (8,8);
			\draw[dashed] (4,0) -- (4,8);
			\draw[dashed] (0,4) -- (8,4);
			
			\quarters{2}{4}{2}{4}
			\quarters{0}{0.5}{3.5}{4}
			\quarters{0.5}{1}{3.5}{4}
			\quarters{1}{1.5}{3.5}{4}
			\quarters{1.5}{2}{3.5}{4}
			\quarters{2.5}{3}{3.5}{4}
			\quarters{3.5}{4}{3.5}{4}
			\quarters{2.5}{3}{2.5}{3}
			\quarters{3.5}{4}{2.5}{3}
			\quarters{3.5}{4}{0}{0.5}
			\quarters{3.5}{4}{0.5}{1}
			\quarters{3.5}{4}{1}{1.5}
			\quarters{3.5}{4}{1.5}{2}
			
			\config{2}{3}{2}{3}{0.25}{blue}
			\config{3}{4}{3}{4}{0.25}{blue}
			\config{2}{3}{3}{4}{0.25}{blue}
			\config{3}{4}{2}{3}{0.25}{blue}
			
			\quarters{6}{8}{2}{4}
			\quarters{4}{4.5}{3.5}{4}
			\quarters{4.5}{5}{3.5}{4}
			\quarters{5}{5.5}{3.5}{4}
			\quarters{5.5}{6}{3.5}{4}
			\quarters{6.5}{7}{3.5}{4}
			\quarters{7.5}{8}{3.5}{4}
			\quarters{6.5}{7}{2.5}{3}
			\quarters{7.5}{8}{2.5}{3}
			\quarters{7.5}{8}{0}{0.5}
			\quarters{7.5}{8}{0.5}{1}
			\quarters{7.5}{8}{1}{1.5}
			\quarters{7.5}{8}{1.5}{2}
			
			\config{6}{7}{2}{3}{0.25}{blue}
			\config{7}{8}{3}{4}{0.25}{blue}
			\config{6}{7}{3}{4}{0.25}{blue}
			\config{7}{8}{2}{3}{0.25}{blue}
			
			\quarters{2}{4}{6}{8}
			\quarters{3.5}{4}{4}{4.5}
			\quarters{3.5}{4}{4.5}{5}
			\quarters{3.5}{4}{5}{5.5}
			\quarters{3.5}{4}{5.5}{6}
			\quarters{3.5}{4}{6.5}{7}
			\quarters{3.5}{4}{7.5}{8}
			\quarters{2.5}{3}{6.5}{7}
			\quarters{2.5}{3}{7.5}{8}
			\quarters{0}{0.5}{7.5}{8}
			\quarters{0.5}{1}{7.5}{8}
			\quarters{1}{1.5}{7.5}{8}
			\quarters{1.5}{2}{7.5}{8}
			
			\config{2}{3}{6}{7}{0.25}{blue}
			\config{3}{4}{7}{8}{0.25}{blue}
			\config{3}{4}{6}{7}{0.25}{blue}
			\config{2}{3}{7}{8}{0.25}{blue}
			
			\quarters{6}{8}{6}{8}
			\quarters{7.5}{8}{4}{4.5}
			\quarters{7.5}{8}{4.5}{5}
			\quarters{7.5}{8}{5}{5.5}
			\quarters{7.5}{8}{5.5}{6}
			\quarters{7.5}{8}{6.5}{7}
			\quarters{7.5}{8}{7.5}{8}
			\quarters{6.5}{7}{6.5}{7}
			\quarters{6.5}{7}{7.5}{8}
			\quarters{4}{4.5}{7.5}{8}
			\quarters{4.5}{5}{7.5}{8}
			\quarters{5}{5.5}{7.5}{8}
			\quarters{5.5}{6}{7.5}{8}
			
			\config{6}{7}{6}{7}{0.25}{blue}
			\config{7}{8}{7}{8}{0.25}{blue}
			\config{6}{7}{7}{8}{0.25}{blue}
			\config{7}{8}{6}{7}{0.25}{blue}
			
			\config{0}{4}{0}{4}{0.25}{blue}
			\config{4}{8}{4}{8}{0.25}{blue}
			\config{0}{4}{4}{8}{0.25}{blue}
			\config{4}{8}{0}{4}{0.25}{blue}
			
			\node at (0.75,0.75)[color=blue]{\LARGE$\SSS_1$};
			\node at (4.75,0.75)[color=blue]{\LARGE$\SSS_2$};
			\node at (0.75,4.75)[color=blue]{\LARGE$\SSS_3$};
			\node at (4.75,4.75)[color=blue]{\LARGE$\SSS_4$};
			
			\end{tikzpicture}
			\vspace{-1em}
			\caption{Covering of $S$ after splitting into four congruent rectangles $S^*$.}
			\label{Pic4.3}
			\vspace{-1.5em}
		\end{figure}
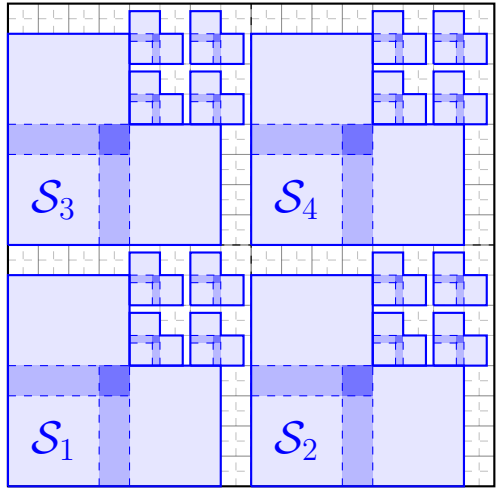
	\end{center}
	
	The final result of this section provides a guideline for constructing differentiation bases to which Theorem~\ref{T2} combined with Lemma~\ref{L0} and Proposition~\ref{P2} can be successfully applied. This construction will play a crucial role in the proof of Theorem~\ref{T1}.  
	
	\begin{proposition} \label{P3}
		Fix a rectangle $U \subseteq \tom$. Let $(d_j)_{j \in \NN}$ and $(m_j)_{j \in \NN}$ be sequences of positive integers and $(\varepsilon_j)_{j \in \NN}$ be a sequence of numbers taking values in $(0,\frac{1}{2}]$. Then there exists a family $\SSS$ of the form $\bigcup_{j \in \NN} \bigcup_{n \in \NN} \SSS_{j,n}$ which satisfies the following properties.
		\begin{enumerate}
			\item[\namedlabel{A1}{\rm(A1)}] Each $\SSS_{j,n}$ is an $(\varepsilon_j, d_j)$-configuration around some $Q_{j,n} \in \RRR_0$.
			\item[\namedlabel{A2}{\rm(A2)}] For every $j \in \NN$, the rectangle $U$ can be covered, up to a set of measure zero, by a~family $\{E_{j,n} : n \in \NN\}$ of disjoint sets $E_{j,n} \coloneqq \bigcup_{Q \in \SSS_{j,n}} Q$. 
			\item[\namedlabel{A3}{\rm(A3)}] For every $j_1, j_2, n_1, n_2 \in \NN$ such that $j_1 < j_2$, if $Q_1 \in \SSS_{j_1,n_1}$ and $Q_2 \in \SSS_{j_2,n_2}$, then either $Q_1 \cap Q_2 = \emptyset$ or $Q_2 \subseteq Q_1$.
			\item[\namedlabel{A4}{\rm(A4)}] There exists a family $\{\Lambda_j : j \in \NN\}$ of subsets $\Lambda_j \subseteq \NN$ such that the following holds. Set $F_j \coloneqq \bigcup_{n \in \Lambda_j} Q_{j,n}$ and $F^*_j \coloneqq \bigcup_{n \in \Lambda_j} E_{j,n}$ for every $j \in \NN$. Then the sets 
			$F^*_j$ are mutually independent events with probability $2^{-m_j}$ with respect to the uniform probability measure on $U$, that is,
			$|\bigcap_{j \in J} F^*_{j}| = |U| \prod_{j \in J} 2^{-m_{j}}$
			for every finite and nonempty subset $J \subseteq \NN$. 
		\end{enumerate}
	\end{proposition}
	
	\begin{proof}
		We construct $\SSS$ inductively with respect to $j \in \NN$ (see Figure~\ref{Pic4.4}). 
		
		In the first step, we apply Lemma~\ref{L1} with $d = d_1$, $\varepsilon = \varepsilon_1$, $m = m_1$, and $U$ as above. We obtain a family $\{\SSS_{1,n} : n \in \NN\}$ of $(\varepsilon_1, d_1)$-configurations such that the sets $E_{1,n}$ are disjoint and cover $U$ up to a set of measure zero. Moreover, $|E_{1,n_1}| = |E_{1,n_2}|$ whenever $\lceil \frac{n_1}{2^{m_1}}\rceil = \lceil \frac{n_2}{2^{m_1}} \rceil$. We set $\Lambda_1 \coloneqq \{i2^{m_1} : i \in \NN\}$ so that $|F_1^*| = |U| 2^{-m_1}$. Note that one can split $U$, up to a set of measure zero, into countably many atoms $U_1$, that is, the minimal nonempty sets in the $\sigma$-algebra $\Sigma_1$ of subsets of $U$ generated by all elements of the $(\varepsilon_1, d_1)$-configurations $\SSS_{1,n}$. Furthermore, such atoms are rectangles in $\tom$.  
		
		Next, suppose that, for some $j_0 \in \NN$, the families $\{\SSS_{1,n} : n \in \NN\}, \dots,\{\SSS_{j_0,n} : n \in \NN\}$ and the subsets $\Lambda_1, \dots, \Lambda_{j_0}$ have been defined and the following properties hold:
		\begin{itemize}
			\item $|\bigcap_{j \in J} F^*_{j}| = |U| \prod_{j \in J} 2^{-m_{j}}$
			for every nonempty subset $J \subseteq \{1, \dots, j_0\}$,
			\item for every $j \in \{1, \dots, j_0\}$, one can split $U$, up to a set of measure zero, into countably many atoms from the $\sigma$-algebra $\Sigma_j$ generated by the $(\varepsilon_j, d_j)$-configurations $\SSS_{j,n}$,
			\item for every $j_1, j_2 \in \{1,\dots,j_0\}$, if $j_1 \leq j_2$ and $U_{j_1} \in \Sigma_{j_1}, U_{j_2} \in \Sigma_{j_2}$ are atoms, then 
			\[
			U_{j_1} \cap U_{j_2} = \emptyset
			\quad \text{or} \quad U_{j_2} \subseteq U_{j_1}.
			\]
		\end{itemize}
		In order to define $\{\SSS_{j_0+1,n} : n \in \NN\}$ and $\Lambda_{j_0+1}$, we proceed as follows. For every atom $U_{j_0} \in \Sigma_{j_0}$, we apply Lemma~\ref{L1} with $d = d_{j_0+1}$, $\varepsilon = \varepsilon_{j_0+1}$, $m = m_{j_0+1}$, and $U_{j_0}$ in place of $U$. We obtain a family $\{\SSS_{j_0+1,n}^{U_{j_0}} : n \in \NN\}$ of $(\varepsilon_{j_0+1}, d_{j_0+1})$-configurations such that the sets $E_{j_0+1,n}^{U_{j_0}}$ are disjoint and cover $U_{j_0}$ up to a set of measure zero. Moreover, $|E_{j_0+1,n_1}^{U_{j_0}}| = |E_{j_0+1,n_2}^{U_{j_0}}|$ whenever $\lceil n_1 / 2^{m_{j_0+1}} \rceil = \lceil n_2 / 2^{m_{j_0+1}} \rceil$. Consider the whole family 
		\[
		\{\SSS_{j_0+1,n}^{U_{j_0}} : n \in \NN, \, U_{j_0} \in \Sigma_{j_0}\}
		\]
		and let $\{\SSS_{j_0+1,n} : n \in \NN\}$ be its 
		enumeration such that $|E_{j_0+1,n_1}| = |E_{j_0+1,n_2}|$ and both sets belong to the same atom $U_{j_0} \in \Sigma_{j_0}$ whenever $\lceil n_1 / 2^{m_{j_0+1}} \rceil = \lceil n_2 / 2^{m_{j_0+1}} \rceil$. We also set $\Lambda_{j_0+1} \coloneqq \{i2^{m_{j_0+1}} : i \in \NN\}$. It is routine to check that the three properties above hold with $\{1,\dots,j_0\}$ replaced by $\{1, \dots, j_0+1\}$.
		
		Let $\SSS \coloneqq \bigcup_{j \in \NN} \bigcup_{n \in \NN} \SSS_{j,n}$. Again, it is routine to verify conditions~\ref{A1}--\ref{A4}. 
	\end{proof}
	
	\begin{center}
		\begin{figure}[H]
			\begin{tikzpicture}
			[scale=1]
			
			\quarters{0}{2}{0}{2}
			
			\quarters{0}{1}{2}{3}
			\quarters{1}{2}{2}{3}
			\quarters{2}{3}{2}{3}
			\quarters{2}{3}{1}{2}
			\quarters{2}{3}{0}{1}
			
			\quarters{0.5}{1}{0.5}{1}
			\quarters{1.5}{2}{0.5}{1}
			\quarters{0.5}{1}{1.5}{2}
			\quarters{1.5}{2}{1.5}{2}
			
			\quarters{0}{2}{4}{6}
			\quarters{2}{4}{4}{6}
			\quarters{4}{6}{0}{2}
			\quarters{4}{6}{2}{4}
			
			\quarters{0}{1}{6}{7}
			\quarters{1}{2}{6}{7}
			\quarters{2}{3}{6}{7}
			\quarters{3}{4}{6}{7}
			
			\quarters{6}{7}{0}{1}
			\quarters{6}{7}{1}{2}
			\quarters{6}{7}{2}{3}
			\quarters{6}{7}{3}{4}
			
			\quarters{0}{1}{3}{4}
			\quarters{1}{2}{3}{4}
			\quarters{2}{3}{3}{4}
			\quarters{3}{4}{3}{4}
			\quarters{3}{4}{2}{3}
			\quarters{3}{4}{1}{2}
			\quarters{3}{4}{0}{1}
			
			\draw[very thick, blue] (0,0) rectangle (3,3);
			\draw[very thick, blue] (0,3) rectangle (3,4);
			\draw[very thick, blue] (3,0) rectangle (4,3);
			\draw[very thick, blue] (3,3) rectangle (4,4);
			\draw[very thick, blue] (0,4) rectangle (4,7);
			\draw[very thick, blue] (4,0) rectangle (7,4);
			
			\config{0}{8}{0}{8}{0.25}{blue}
			
			\config{0}{1}{0}{1}{0.35}{red}
			\config{1}{2}{0}{1}{0.35}{red}
			\config{0}{1}{1}{2}{0.35}{red}
			\config{1}{2}{1}{2}{0.35}{red}
			
			\config{0}{0.5}{2}{2.5}{0.35}{red}
			\config{0}{0.5}{2.5}{3}{0.35}{red}
			\config{0.5}{1}{2}{2.5}{0.35}{red}
			\config{0.5}{1}{2.5}{3}{0.35}{red}
			\config{1}{1.5}{2}{2.5}{0.35}{red}
			\config{1}{1.5}{2.5}{3}{0.35}{red}
			\config{1.5}{2}{2}{2.5}{0.35}{red}
			\config{1.5}{2}{2.5}{3}{0.35}{red}
			\config{2}{2.5}{2}{2.5}{0.35}{red}
			\config{2}{2.5}{2.5}{3}{0.35}{red}
			\config{2.5}{3}{2}{2.5}{0.35}{red}
			\config{2.5}{3}{2.5}{3}{0.35}{red}
			\config{2}{2.5}{1}{1.5}{0.35}{red}
			\config{2}{2.5}{1.5}{2}{0.35}{red}
			\config{2.5}{3}{1}{1.5}{0.35}{red}
			\config{2.5}{3}{1.5}{2}{0.35}{red}
			\config{2}{2.5}{0}{0.5}{0.35}{red}
			\config{2}{2.5}{0.5}{1}{0.35}{red}
			\config{2.5}{3}{0}{0.5}{0.35}{red}
			\config{2.5}{3}{0.5}{1}{0.35}{red}
			
			\end{tikzpicture}
			\vspace{-1em}
			\caption{Second step of the iterative construction of $\SSS$ in $\TT^2 \times \TT^{2,\omega}$.}
			\label{Pic4.4}
			\vspace{-1.5em}
		\end{figure}
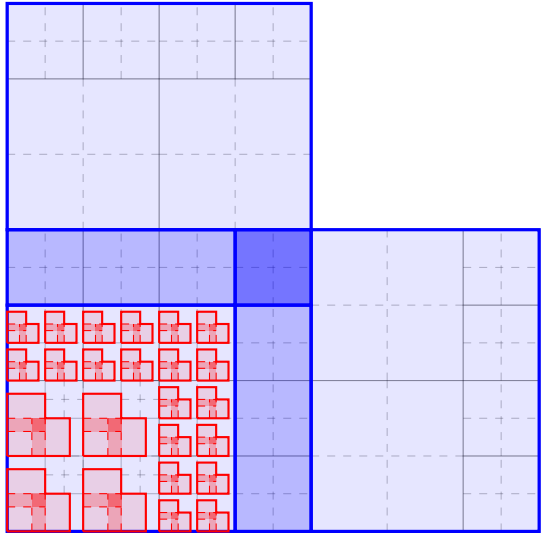
	\end{center}
	
	\section{Proofs of Theorem~\ref{T0} and Theorem~\ref{T1}} \label{S5}
	
	In the last section, we gather all previously obtained results to prove our main theorems. 
	
	\begin{proof}[Proof of Theorem~\ref{T0}]
		First, let us construct $\BBB_{\geq}$. If $p_0 = 1$, then we can take $\BBB_{\geq} = \RRR_0$. Indeed, we have ${\rm diff}(\RRR_0) = [1,\infty]$ by \cite[Corollary~16]{FR20}. For $p_0 \in (1,\infty)$, let $\BBB_{\geq}$ be the family $\SSS$ from Proposition~\ref{P3} applied with 
		\[
		U = \tom, \quad d_j = j, \quad \varepsilon_j = j^{-1/p_0} / 2 
		\]
		and $m_j \in \NN$ being the unique positive integer such that 
		\[
		j^{-2}/4 < 2^{-m_j} (1+d_j-\varepsilon_j d_j)^{-1} \leq j^{-2}/2.
		\]
		We notice that if $j \in \NN$ and $\varepsilon \in (0,\frac{1}{2}]$, then
		\[
		j^{-2} (1+j-\varepsilon j) / 2 < j^{-2} (1+j) / 2 \leq j^{-1} \leq 1, 
		\]
		thanks to which such a choice of $m_j \in \NN$ is possible. 
		Then
		$
		\lim_{j \to \infty} \varepsilon_j d_j^{1/p_0} < \infty
		$ 
		and so $\MMM_{\BBB_{\geq}}$ is of weak-type $(p_0,p_0)$ by conditions~\ref{A1}--\ref{A3} together with Proposition~\ref{P2} and Theorem~\ref{T2}. Therefore, $p_0 \in {\rm diff}(\BBB_{\geq})$ by Lemma~\ref{L0} (for the density assumption, see \cite[Proposition~7.9]{Fo99}), which implies $[p_0, \infty] \subseteq {\rm diff}(\BBB_{\geq})$ by Proposition~\ref{P1}. 
		
		Next, choose $p \in [1,p_0)$ 
		and define
		\begin{align} \label{def-f}
			f \coloneqq \sup_{j \in \NN} f_j,
		\end{align}
		where
		$f_j \coloneqq \varepsilon_j^{-1} \ind{F_j}$ for $F_j$ from \ref{A4}.  
		Then $f \in L^p(\tom)$, since
		\[
		\| f \|_{p}^p \leq \sum_{j \in \NN} \|f_j\|_{p}^p
		= \sum_{j \in \NN}  \varepsilon_j^{-p} 2^{-m_j} (1+d_j-\varepsilon_j d_j)^{-1}
		\leq 2^{p-1} \sum_{j \in \NN} j^{p/p_0} j^{-2} < \infty. 
		\]
		In order to justify the first inequality above, set $G \coloneqq \{x \in \tom : |f_1(x)| > |f_2(x)|\}$. Then
		\[
		\| \max\{f_1,f_2\} \|_p^p = \int_G |f_1(x)|^p \, {\rm d}x + \int_{\tom \setminus G} |f_2(x)|^p \, {\rm d}x
		\leq \| f_1 \|_p^p + \| f_2 \|_p^p.
		\]
		By iterating this and passing to the limit on both sides, we obtain $\| f \|_{p}^p \leq \sum_{j \in \NN} \|f_j\|_{p}^p$. Next, note that by the second Borel--Cantelli lemma, for almost all $x \in \tom$, the set 
		\[
		\{ j \in \NN : x \in E_{j,n} \text{ for some } n \in \Lambda_j\}
		\]
		is infinite. Indeed, the sets $F^*_j$ are mutually independent events and, due to the choice $d_j=j$ and $\varepsilon_j = j^{-1/p_0} / 2$,
		\[
		\sum_{j \in \NN} |F_j^*| \geq \sum_{j \in \NN} j^{-2} (1+d_j-\varepsilon_j d_j) / 4 = \infty.
		\]
		Now, $\overline{f}(x) \geq 1$ holds for each such $x$. Indeed, if $x \in E_{j,n}$ and $n \in \Lambda_j$, then ${\rm Avg}_f(Q) \geq {\rm Avg}_{f_j}(Q) \geq 1$ for some $x \in Q \in \SSS_{j,n}$. Since $f$ vanishes outside $\bigcup_{j \in \NN} F_j$ and 
		\[
		\sum_{j \in \NN} |F_j| \leq \sum_{j \in \NN} j^{-2}/2 < 1,
		\]
		we have $\overline{f}(x) > f(x)$ on a set with positive measure. Consequently, $p \notin {\rm diff}(\BBB_{\geq})$. Since $p \in [1,p_0)$ was arbitrary, we obtain ${\rm diff}(\BBB_{\geq}) \subseteq [p_0, \infty]$.
		
		Similarly, let us construct $\BBB_{>}$. For $p_0 \in [1,\infty)$, apply Proposition~\ref{P3} with 
		\[
		U = \tom, \quad d_j = j, \quad \varepsilon_j = (j \log^{-2}(j+1))^{-1/p_0} / 2
		\]
		and $m_j \in \NN$ being the unique positive integer such that
		\[
		j^{-2} / 4 < 2^{-m_j} (1+d_j-\varepsilon_j d_j)^{-1} \leq j^{-2}/2.
		\]
		As before, such a choice is possible if $\varepsilon_j \in (0,\frac{1}{2}]$. To see this, note that
		$\sqrt{x} \geq \log(x+1)$ for $x \in [0,\infty)$, since both functions vanish at $x=0$ and $(2\sqrt{x})^{-1} \geq (x+1)^{-1}$ for $x \in (0,\infty)$. 
		Then 
		$
		\lim_{j \to \infty} \varepsilon_j d_j^{1/p} < \infty
		$
		for every $p \in (p_0, \infty)$ which implies 
		$(p_0, \infty] \subseteq {\rm diff}(\BBB_{>})$. 
		
		Next, consider 
		$f$ defined as in \eqref{def-f}. Then $f \in L^{p_0}(\tom)$ and, as before, we have $\overline{f}(x) > f(x)$ on a set with positive measure. Thus, $p_0 \notin {\rm diff}(\BBB_{>})$ so that ${\rm diff}(\BBB_{>}) \subseteq (p_0, \infty]$. 
	\end{proof}
	
	Before proving Corollary~\ref{Cor1}, we provide an additional remark concerning Theorem~\ref{T0}.
	
	\begin{remark} \label{remCHAR}
		In fact, condition~\ref{A4} is not essential---it merely serves to simplify the proof of Theorem~\ref{T0} by allowing us to invoke the second Borel--Cantelli lemma. To see this, assume that $\BBB$ satisfies conditions~\ref{A1}--\ref{A3} and fix $p \in (1,\infty)$. If the quantities $\varepsilon_j d_j^{1/p}$ are uniformly bounded in $j$, then, as before, $\MMM_\BBB$ is of weak-type $(p,p)$ and $p \in {\rm diff}(\BBB)$. In turn, assume that $\varepsilon_j^p d_j \geq j + N_0$ for all $j \in \NN$ and a large constant $N_0 \in \NN$, passing to a subsequence if needed.  
		For every $j_1, j_2 \in \NN$ such that $j_1 < j_2$, each set $E_{j_2,n}$ must be contained in some atom $U_{j_1} \in \Sigma_{j_1}$ in order for condition~\ref{A3} to hold. Thus, 
		\[
		|E_{j_2,n}| \leq (1+d_{j_2}-\varepsilon_{j_2} d_{j_2})(2-\varepsilon_{j_2})^{-d_{j_2}} |U_{j_1}|
		\]
		because $U_{j_1}$ is a rectangle with sides parallel to the coordinate axes. Moreover,
		\[
		(1+d_{j}-\varepsilon_{j} d_{j})(2-\varepsilon_{j})^{-d_{j}}
		\leq (2d_{j}) (2/3)^{d_{j}}
		\leq d_j^{-1}/8 \leq j^{-1}/8  
		\]
		if $N_0$ is large enough. Thus, there are sets $E_{j_2,n_1}, \dots, E_{j_2,n_k}$ whose union covers a portion of $U_{j_1}$ of size $\alpha |U_{j_1}|$ for $j_2^{-1}/8 \leq \alpha \leq j_2^{-1}/4$. The exact value of $\alpha$, however, may depend on the atom $U_{j_1}$; hence we cannot apply the second Borel--Cantelli lemma directly, but must instead mimic its proof. Consider 
		$f$ defined as in \eqref{def-f} with each $\Lambda_j$ replaced by the set of indices chosen to cover a suitable portion of each atom. It remains to verify that:
		\begin{itemize}
			\item $f$ belongs to $L^p(\tom)$ by $\sum_{j \in \NN} \varepsilon_j^{-p} j^{-1} (1+d_j-\varepsilon_j d_j)^{-1} / 4 < \infty$,
			\item $f$ vanishes on a set with positive measure by $\sum_{j \in \NN} j^{-1} (1+d_j-\varepsilon_j d_j)^{-1} / 4 < 1$,
			\item $\overline{f}(x) \geq 1$ almost everywhere by $\sum_{j \in \NN} j^{-1} / 8 = \infty$.
		\end{itemize}  
		
		To conclude, for every $p \in (1,\infty)$, one can prove the following characterization 
		\[
		p \in {\rm diff}(\BBB)
		\iff
		\sup \{\varepsilon_j d_j^{1/p} : j \in \NN\} < \infty  
		\]
		if $\BBB$ satisfies conditions~\ref{A1}--\ref{A3} (regardless of whether condition~\ref{A4} is satisfied). 
	\end{remark}
	
	\begin{proof} [Proof of Corollary~\ref{Cor1}]
		Let $\BBB$ be any basis in $\tom$ constructed in the proof of Theorem~\ref{T0}. We shall find a Busemann--Feller basis $\BBB^*$ in the unit interval $[0,1]$, whose elements are finite unions of open intervals, such that ${\rm diff}(\BBB) = {\rm diff}(\BBB^*)$. If $\BBB = \RRR_0$, then we can take the standard dyadic basis in $[0,1]$, so assume that $\BBB$ satisfies conditions~\ref{A1}--\ref{A4}. Fix $j=1$ and split $[0,1]$, up to a set of measure zero, into open intervals $E_{1,n}^*$ such that the length of each $E_{1,n}^*$ is precisely $|E_{1,n}|$. Then split each $E_{1,n}^*$ into open subintervals $U_1^*$, with lengths corresponding to the sizes of atoms $U_1 \in \Sigma_1$ determined by $\SSS_{1,n}$. Next, define $E_{2,n}^*$ and $U_2^*$ analogously, dealing with each subinterval $U_1^*$ separately, and continue in this manner for every $j \in \NN$. Let $S^* \subseteq [0,1]$ be an element of $\BBB^*$ if and only if $S^*$ is a finite union of open intervals $U_j^*$ for some $j \in \NN$, and the corresponding atoms $U_j \in \Sigma_j$ determine a rectangle $S \in \BBB$.
		It is routine to show that ${\rm diff}(\BBB) = {\rm diff}(\BBB^*)$. Namely, we either deduce the weak-type $(p,p)$ bounds for $\MMM_{\BBB^*}$ from the corresponding bounds for $\MMM_\BBB$, or construct a specific function $f^*$ corresponding to the function $f$ defined in \eqref{def-f}. 
	\end{proof}
	
	\begin{proof}[Proof of Theorem~\ref{T1}]
		If $\XXX$ is complete, then ${\rm diff}(\BBB)$ must take one of the forms \ref{C1}--\ref{C6} by Proposition~\ref{P1}. If instead $X = \bigcup_{n \in \NN} G_n$ for open sets $G_n \subseteq X$ satisfying $\mu(G_n) < \infty$, then ${\rm diff}(\BBB)$ must take one of the forms \ref{C1}--\ref{C4}, again by Proposition~\ref{P1}.
		
		Conversely, fix $p_0 \in [1,\infty)$ and let $P \subseteq [1,\infty]$ take one of the forms \ref{C1}--\ref{C4}. We will find a pair $(\XXX,\BBB)$ such that ${\rm diff}(\BBB) = P$ and $\mu(X) < \infty$. In addition, we choose $(\XXX,\BBB)$ so that $\BBB$ is uncentered and $\rho(x,y) \leq 1$ for every $x,y \in X$. 
		
		Before that, however, note that if we can do that, then we can also find a pair $(\XXX,\BBB)$ such that ${\rm diff}(\BBB) = P$ for any $P \subseteq [1,\infty]$ that takes one of the forms \ref{C5}--\ref{C6}. Indeed, assume that $P = [p_0, \infty)$ (resp. $P = (p_0, \infty)$). Let $(\XXX', \BBB')$ be the pair mentioned above for which ${\rm diff}(\BBB') = [p_0, \infty]$ (resp. ${\rm diff}(\BBB') = (p_0, \infty]$). Also, let $(\XXX'', \BBB'')$ be the pair from Example~\ref{E1}. Writing $\XXX' = (X', \rho', \mu')$ and $\XXX'' = (X'', \rho'', \mu'')$, we define $(\XXX, \BBB)$ as follows. Set $X \coloneqq X' \cup X''$, assuming that $X' \cap X'' = \emptyset$. For every $x,y \in X$, let 
		\[
		\rho(x,y) \coloneqq
		\begin{cases}
		\rho'(x,y) & \text{if } x,y \in X', \\
		\rho''(x,y) & \text{if } x,y \in X'', \\
		1 & \text{otherwise.}
		\end{cases}
		\]
		Next, for every $E \subseteq X$ such that $E \cap X'$ is $\mu'$-measurable, set 
		\[
		\mu(E) \coloneqq \mu'(E \cap X') + \mu''(E \cap X'').
		\]
		Then, for $\BBB \coloneqq \BBB' \cup \BBB''$, considered as an uncentered basis, we have 
		\[
		{\rm diff}(\BBB) = {\rm diff}(\BBB') \cap {\rm diff}(\BBB'')
		= {\rm diff}(\BBB') \cap [1,\infty)
		\]
		and so ${\rm diff}(\BBB) = [p_0,\infty)$ (resp. ${\rm diff}(\BBB) = (p_0,\infty)$), as desired. 
		
		It remains to find suitable pairs illustrating the cases~\ref{C1}--\ref{C4}. In what follows, the underlying space $\XXX$ is always the infinite-dimensional torus $\tom$ equipped with the standard metric and measure, or rather its complete version. In particular, we have $\mu(X) < \infty$ and $\rho(x,y) \leq 1$ for every $x,y \in X$. Moreover, each basis $\BBB$ is uncentered and countable, with the sole exception of the basis $\RRR$, which is uncentered and uncountable. 
		
		First, consider the case~\ref{C1}. Then $P = \emptyset$ and so we can take $\BBB = \RRR$. Indeed, we have ${\rm diff}(\RRR) = \emptyset$ by \cite[Theorem~1.1]{Ko21}. 
		Next, consider the cases~\ref{C3}--\ref{C4}. Now, we can take a basis $\BBB$ constructed as in the proof of Theorem~\ref{T0}.   
		Finally, consider the case~\ref{C2}. We split $\tom$ into a family $\{U_k : k \in \NN\}$ of disjoint rectangles. For every $k$, let $\BBB_k$ be defined as in the case~\ref{C3} with $p_0 = k+1$ and $U_k$ instead of $\tom$. Set $\BBB \coloneqq \bigcup_{k \in \NN} \BBB_k$. Then ${\rm diff}(\BBB) = \bigcap_{k \in \NN} {\rm diff}(\BBB_k) = \{\infty\}$, as desired.   
	\end{proof}
	
	We conclude the article with the following remark, which relates the range of differentiation to the range of weak-type boundedness for a given uncentered basis $\BBB$.
	
	\begin{remark} \label{rem3}
		Let $\XXX$ be a metric measure space. Let $\BBB$ be an uncentered differentiation basis in $\XXX$ such that $\MMM_\BBB f $ is measurable for every $p \in [1,\infty]$ and $f \in L^p(\XXX)$. Define 
		\[
		{\rm max}(\BBB) \coloneqq
		\big\{ p \in [1,\infty] : \MMM_\BBB  \text{ satisfies the weak-type } (p,p) \text{ inequality} \big\}
		\]
		with the convention that being of weak-type $(\infty,\infty)$ is equivalent to being bounded on $L^\infty(\XXX)$.
		Then ${\rm max}(\BBB)$ takes one of the forms \ref{C2}--\ref{C4} for some $p_0 \in [1,\infty)$. Moreover, if $\BBB$ is countable and, for every $p \in [1,\infty)$, the space $L^p(\XXX)$ contains a dense subset consisting of continuous functions, then ${\rm max}(\BBB) \setminus \{ \infty \} \subseteq {\rm diff}(\BBB)$ by Lemma~\ref{L0}. 
		
		Conversely, let $P_1 \subseteq [1,\infty]$ take one of the forms \ref{C1}--\ref{C6} for some $p_1 \in [1,\infty)$ in place of $p_0$, and let $P_2 \subseteq [1,\infty]$ take one of the forms \ref{C2}--\ref{C4} for some $p_2 \in [1,\infty)$ in place of $p_0$. If $P_2 \setminus \{ \infty \} \subseteq P_1$, then one can find a pair $(\XXX, \BBB)$ such that ${\rm diff}(\BBB) = P_1$ and ${\rm max}(\BBB) = P_2$. 
		Indeed, let $(\XXX', \BBB')$ be the pair constructed in the proof of Theorem~\ref{T1} for which ${\rm diff}(\BBB') = P_1$ and ${\rm max}(\BBB') = P_1 \cup \{\infty\}$, and let $(\XXX'', \BBB'')$ be the pair appearing in \cite[Corollary~1.3]{KRR23} for which ${\rm diff}(\BBB'') = [1,\infty]$ (this property is not stated explicitly in \cite[Corollary~1.3]{KRR23}, but it follows easily from the construction of $\BBB''$) and ${\rm max}(\BBB'') = P_2$. Define $(\XXX, \BBB)$ as in the proof of Theorem~\ref{T1}. Then 
		\begin{align*}
			{\rm diff}(\BBB) = \, {\rm diff}(\BBB') \, & \cap \, {\rm diff}(\BBB'') \, = P_1,\\
			{\rm max}(\BBB) = {\rm max}(\BBB') & \cap {\rm max}(\BBB'') = P_2.
		\end{align*}
		Therefore, the pair $(\XXX, \BBB)$ satisfies the desired properties. 
	\end{remark}

\end{document}